	\documentclass[12pt,oneside,reqno]{amsart}
	\usepackage{bbm}

        \usepackage{cite} 
        \newcommand{\Noopsort}[1]{} 
 
	\usepackage{amsmath}
	\usepackage{graphicx}
	\usepackage{mathrsfs}
	\usepackage{stmaryrd}
	\usepackage{enumerate,amsmath,amssymb,amsthm}
	\numberwithin{equation}{section}
	\usepackage{amsthm}
	\usepackage{amssymb,amsfonts,mathrsfs}
	\usepackage{enumerate}
	\usepackage{color}
	\usepackage{xcolor}
	\usepackage{cite}

	\usepackage{hyperref}
     \hypersetup{colorlinks=true,
				linkcolor=blue,
				anchorcolor=blue,
				citecolor=blue}
	\theoremstyle{definition}

	\theoremstyle{remark}

	\newtheorem{theorem}{Theorem}[section]
	\newtheorem{lemma}[theorem]{Lemma}
	\newtheorem{remark}[theorem]{Remark}
	\newtheorem{definition}[theorem]{Definition}
	\newtheorem{proposition}[theorem]{Proposition}

      \makeatletter
		\renewcommand{\subsection}{\@startsection
			{subsection} 
			{2} 
			{0mm} 
			{0.5\baselineskip} 
			{0.3\baselineskip} 
			{\normalfont\normalsize\raggedright}} 
		\makeatother
		
	\def\E{\mathbb{E}}
	\def\P{\mathbb{P}}
	\def\R{\mathbb{R}}
	\def\B{\mathcal{B}}
	\def\d{\mathrm{d}}
	\def\N{\mathbb{N}}
	\def\C{\mathbb{C}}

	\def\1{\mathbbm{1}}
	\def\e{\mathrm{e}}
	\def\ito{It$\hat{\mathrm{o}}$}
	\def\gron{Gr$\ddot{\mathrm {o}}$nwall}
        
	\def\hold{H$ \ddot{\mathrm{o}} $lder}

        \def\geq{\geqslant}
        \def\leq{\leqslant}
\allowdisplaybreaks[4]

\begin{document}
\title{The Large Deviation Principle for Stochastic Flow of Stochastic Slow-Fast Motions}
\author{Mingkun Ye$^{\text{a},*}$, and Zuozheng Zhang$^{\text{b}}$}
\thanks{$^{*}$ Corresponding author.}
\thanks{E-mail addresses: mingkunye@foxmail.com(M. Ye), zuozhengzhang@bit.edu.cn(Z. Zhang)}	
	\dedicatory{
        $^{\rm{a}}$School of Mathematics,
		Sun Yat-sen University, \\
		Guangzhou, 510275,  P.R.China\\
		$^{\rm{b}}$School of Mathematics and Statistics,
		Beijing Institute of Technology, \\
		Beijing, 100081,  P.R.China}
	
\begin{abstract}
In this paper, we consider a kind of fully-coupled slow-fast motion, in which the slow variable satisfies the non-Lipschitz condition.
  We prove that the stochastic flow of the slow variable 
  exists and moreover, satisfies the large deviation principle. The argument  is mainly based on Khasminskii’s averaging principle\cite{Khasminskii}, the variational representation of the exponential functional
  of the Brownian motion \cite{DupuisAOP1998}, and the weak convergence framework proposed by Budhiraja and Dupuis \cite{BUDHIRAJA2000}.	
\end{abstract}
\maketitle
\rm
	
\section{Introduction}
Consider the following stochastic slow-fast motion
	\begin{equation}\label{EQ:SDE:01}
		\left\{\begin{array}{l}
			\d X_{t}^{(\varepsilon, \delta)}=f_{1}(X_{t}^{(\varepsilon, \delta)}, Y_{t}^{(\varepsilon, \delta)}) \d t+\sqrt{\varepsilon} \sigma_{1}(X_{t}^{(\varepsilon, \delta)}) \d W^1_{t},\quad X_0^{(\varepsilon,\delta)}=x,\\
			\delta \d Y_{t}^{(\varepsilon, \delta)}=f_{2}(X_{t}^{(\varepsilon, \delta)}, Y_{t}^{(\varepsilon, \delta)}) \d t+\sqrt{\delta} \sigma_{2}(X_{t}^{(\varepsilon, \delta)}, Y_{t}^{(\varepsilon, \delta)}) \d W^2_{t},\quad Y_0^{(\varepsilon,\delta)}=y,
		\end{array}\right.
	\end{equation}
	where
	\( f_{i}: \mathbb{R}^{{d}} \times \mathbb{R}^{{d}} \rightarrow \mathbb{R}^{{d}},i=1,2\),  \(\sigma_{1}: \mathbb{R}^{{d}} \rightarrow \mathbb{R}^{d}\otimes\mathbb{R}^{d}\) and \(\sigma_{2}: \mathbb{R}^{{d}} \times \mathbb{R}^{{d}} \rightarrow \mathbb{R}^{d}\otimes \mathbb{R}^{d} \),
	both are measurable functions. $W:=(W^1,W^2)$ is
 a $2d$-dimensional standard Brownian motions on a probability  space  $(\Omega,\mathscr{F}, \P)$. 
 Let $(\mathscr{F}_t)_{t\geq 0}$ be the corresponding natural filtration with respect to $(W_1,W_2)$. The two small positive parameters  $ \varepsilon  $ and  $ \delta  $ satisfy  $ \delta=o (\varepsilon) $, which are used to describe the time scale separation between the slow variable $ X^ {( \varepsilon,  \delta)}  $ and the fast variable $ Y^ {( \varepsilon,  \delta)} $. 
 
 The large deviation principle (LDP to be short) has been extensively studied in many stochastic dynamical systems.  A powerful approach for studying the LDP is the well-known weak convergence method (see, for example, Dupuis and Ellis \cite{DUPUIS2011}, Veretennikov \cite{VERETENNIKOV2000}, Budhiraja and Dupuis \cite{BUDHIRAJA2000}). 
Moreover, this approach has been widely applied in various multi-scale diffusion processes.
Dupuis and Spiliopoulos \cite{DUPUIS-SPA-2012} build the LDP for locally periodic SDEs with small noise and fast oscillating coefficients.
Sun et al. \cite{Sun2021} establish the LDP for SPDEs with slow and fast time-scales, where the slow component is a one-dimensional stochastic Burgers equation with small noise and the fast component is a stochastic reaction-diffusion equation.
 Hong et al. \cite{HONG2021SPDE,hongLargeDeviationPrinciple2024} treat with  multi-scale stochastic partial differential equations (SPDEs to be short) with locally monotone condition and fully cocal monotone condition respectively. Hong et al. \cite{HONG2021} investigate the LDP of a kind of McKean-Vlasov SPDE slow-fast motion.
Hu et al.  \cite{HU2019} deal with the systems of slow-fast stochastic reaction-diffusion equation. Li and Liu \cite{LIGE2022} establish the LDP for the synchronized system using the similar argument like \cite{DUPUIS-SPA-2012}. 
Inahama et al. \cite{Xuyong2023Arxiv} provide the corresponding  LDP for slow-fast motions driven by mixed fractional Brownian motion.
As for the stochastic flow of SDEs, Ren and Zhang \cite{REN2003,REN2005,REN2006}, Zhang\cite{ZHANG2005,ZHANG2010}, present rich results about stochastic flows of SDEs under non-Lipshictz coefficients. Moreover, Ren and Zhang \cite{REN2005Freidlin,REN2008Freidlin} prove the LDP of stochastic flow about SDEs and SPDEs respectively.

Based on these extensive research results about the LDP and the stochastic flow, we consider the LDP of the stochastic flow about a kind of stochastic slow-fast motion defined by \eqref{EQ:SDE:01} based on the averaging principle and the weak convergence method. To be concrete,  we treat the slow variable $ X^{( \varepsilon,  \delta)}  $  as random variables taking value in the  specific function space $ C([0,T];C(\R^d;\R^d))  $which is identified as $ C([0,T]\times \R^d;\R^d).$ In other words, the slow variable $X_{\cdot}^{(\varepsilon,\delta)}(\cdot,y_0)$ can be regarded as a stochastic field with index  set $ [0,T]\times \R^d $.  The purpose of this paper is to prove that under some suitable conditions, the stochastic flow of slow variable  in \eqref {EQ:SDE:01} with  fixed fast variable initial value $y_ 0$, 
	\[ 
	[0,T]\times \R^d \ni (t,x) \mapsto X_ {t}^{( \varepsilon,  \delta)} (x, y_0)\in \R^d
	 \]
	 exists, and further satisfies the LDP, i.e., there is  a rate function $ I $ such that for any set $ A\in\mathcal{B}(C([0,T]\times \R^d;\R^d)) ,$
	 \[ 
	\begin{split}
	     -\inf_{\varphi\in A^{\circ}}I(\varphi)& \leq \liminf_{\varepsilon\to 0}\varepsilon \log \P(X_ {t} ^ {( \varepsilon,  \delta)} (x, y_0)\in A) \\
     &  \leq \limsup_{\varepsilon\to 0}\varepsilon \log \P(X_ {t} ^ {( \varepsilon,  \delta)} (x, y_0)\in A) \leq-\inf_{\varphi\in\bar{A}}I(\varphi),
	\end{split}
	 \]
	 where $A^{\circ}$, $\bar{A} $ represent the interior and closure of set $A $ respectively.

     This paper is organized as follows. In Section \ref{SEC02}, we prepare some notions and necessary assumptions, and propose the main result of this article. In Section \ref{SEC03}, we prove some auxiliary lemmas, including detailed moment estimations for the frozen process and the controlled system.  In Section \ref{SEC04}, we give the tightness analysis of slow variables. Based on this result, we further prove the Laplace principle which implies the LDP. 
     
    Throughout the paper,  $C$ with or without indexes will denote different constants (depending on the indexes) whose values are not important.
	\section{Preliminaries and Main Results}\label{SEC02}
 Given $T,N>0,$ we denote by
	\[ 
	\mathbb{C}_T:=C([0,T]\times \R^d;\R^d),~\mathbb{C}_T^N:=C([0,T]\times D_N;\R^d),
	\] 
	where $ D_N=\{x\in \R^d:|x|\leq N\}. $
	One can prove that $ \mathbb{C}_T^N $ is a Banach space, its  uniform norm is
	\[ 
	\|f\|_{\mathbb{C}_T^N}:=\sup _{(t, x) \in[0, T] \times D_N}|f(t, x)|,
	\]
	and $ \mathbb{C}_T $ is a Polish space that equiped with semi-norm
	\[ 
	\|f\|_{\mathbb{C}_T}:=\sum_{N=1}^{\infty} 2^{-N}(\|f\|_{\mathbb{C}_T^N} \wedge 1).
	\]
	Let	$ \mathcal{B}(\mathbb{C}_T) $ denote the $ \sigma $-algebra generated by the topology induced by the semi-norm. Define $ \mathbb{H} $ by
	\[ 
	\mathbb{H}:= \left\{h:[0,T]\to\R^{2d};  \int_{0}^{T}|\dot{h}(s)|^2\d s<\infty\right\},
	\]
 where $\dot{h}(s)$ denote the gradient of $h(s)(x)$ with respect to $x$. It is well known that $ \mathbb{H} $ is a Hilbert space with norm $ \|h\|_{\mathbb{H}}^2 := \int_{0}^{T}|\dot{h}(s)|^2\d s $, which is called Cameron-Martin Space. For any  \( N>0 \), define 
	$
	B_{N}:=\left\{h \in \mathbb{H}:\|h\|_{\mathbb{H}} \leq N\right\},
	$
	which can be  metrizable as a compact Polish space under the weak topology in $ \mathbb{H} $.  Set
	\[
	\mathcal{A}:=\left\{h:[0,T]\times \Omega \rightarrow \mathbb{R}^{2d}; h \text{ is } \mathscr{F}_{t} \text {-predictable process and }\mathbb{E} \int_{0}^{T}|\dot{h}(s)|^{2} \d s<\infty\right\}
	\]
	and 
	\[
	\mathcal{A}_{N}:=\left\{h \in \mathcal{A}: h \in B_{N} \text {, a.s.}\right\} .
	\]
 
	Next, we recall the definition of the LDP. Suppose that \( E \) is a Polish space and \( \mathcal{B}(E) \) is therein a Borel \( \sigma \)-algebra. 
	
	\begin{definition}
		A function $ I $ mapping $ E  $ to $ [0,\infty] $ is called a rate function if for each $ a < \infty $, the level set $\{x \in E : I(x)\leq a\} $ is compact set. For any $ A\in \mathcal{B}(E), $ we set that $ I(A):=\inf_{x\in A}I(x). $
	\end{definition}
	
	\begin{definition}[LDP]
		If $ I $ is a rate function on Polish space $ E $, $ \{X^\varepsilon\} $ is a famliy of  $ E $-valued random variables on $ (\Omega,\mathscr{F},\P) $, we call that  $ \{X^\varepsilon\} $ satisfies the LDP on $ E $ with rate function $ I $,  if the following two conditions hold:\\
		(i) Large deviation upper bound. For any closed set $ F\in \B(E), $ 
		\[ 
		\limsup_{\varepsilon \to 0}\varepsilon \log \P(X^\varepsilon\in F) \leq -I(F).
		\]
		(ii) Large deviation lower bound . For any open set $ G\in \B(E), $
		\[ 
		\liminf_{\varepsilon \to 0}\varepsilon \log \P(X^\varepsilon\in G) \geq -I(G).
		\]
	\end{definition}
	Usually, when dealing with concrete problems, we are more concerned with the exponential estimation of functions rather than the indicative functions of open or closed sets, which has led to the study of the Laplace principle. The well-known result is that the LDP is equivalent to the Laplace principle. So we only need to prove that the target solution process satisfies the corresponding Laplace principle. The definition of the Laplace principle is as follows.
	\begin{definition}[Laplace principle]
		Suppose that $ E  $ is a Polish space, the set $ \{X^{\varepsilon}\} $ is a famliy of  $ E $-valued random variable on probability space $ (\Omega,\mathscr{F},\P) $, and $ I $ is a rate function on $ E $. We say that $ \{X^{\varepsilon}\} $ satisfies the Laplace principle on $ E $ with the rate function $ I  $,  if for all real bounded continuous functions $ f:E\rightarrow \R, $ 
		\[ 
		\lim _{\varepsilon \rightarrow 0}-\varepsilon \log \mathbb{E}\left[\exp \left\{-\frac{f\left(X^{\varepsilon}\right)}{\varepsilon}\right\}\right]=\inf _{x \in E}[I(x)+f(x)].
		\]
	\end{definition}

%
%
%

We prepare some lemmas which will be needed in the sequel. The first one is the following generalization of the Gronwall-Bellman type inequality which comes from \cite{bihari1956generalization}. 

\begin{lemma}[Bihari's inequality]\label{LEM:0801:02}
     Let \( \rho: \mathbb{R}_{+} \mapsto \mathbb{R}_{+} \) be a continuous and non-decreasing function. If \( f(t), q(t) \) are two strictly positive functions on \( \mathbb{R}_{+} \) such that
\[
f(t) \leqslant f(0)+\int_{0}^{t} q(s) \rho(f(s)) \d s, \quad t \geqslant 0,
\]
then
\[
f(t) \leqslant g^{-1}\left(g(f(0))+\int_{0}^{t} q(s) \d s\right),
\]
where \( g(x):=\int_{x_{0}}^{x} \frac{1}{\rho(y)} \d y \) is well defined for some \( x_{0}>0 \).
\end{lemma}
In what follows, we give the following concave function,
	\[
	\rho_{\eta}(x):=\left\{\begin{array}{ll}
		x \log x^{-1}, & x \leq \eta, \\
		\eta \log \eta^{-1}+(\log \eta^{-1}-1)(x-\eta), & x>\eta,
	\end{array}\right.
	\]
	where \( 0<\eta<1 / \e \). The undermentioned Bihari-type inequalities can be found in  Zhang \cite[Lemma 2.1]{ZHANG2010}, Ren and Zhang \cite[Lemma 1.1]{REN2005Freidlin}.
	\begin{lemma}[Bihari-type inequality]\label{THM02}
 (1) Let \( f(t) \) be a strictly positive function on \( \mathbb{R}_{+}\) satisfying for some \( \delta>0, \)
\[
f(t) \leq f(0)+\delta \int_{0}^{t} \rho_{\eta}(f(s)) \mathrm{d} s, \quad t \geq 0,
\]
then for any \( T>0 \), there is a constant \( C=C(T, \delta, \eta) \) such that
\[
f(t) \leq C\left(f(0)+(f(0))^{\exp \{-\delta T\}}\right), \quad t \in[0, T] .
\]
(2) Let \( f(t), q(t) \) be two strictly positive functions on \( \mathbb{R}^{+} \) satisfying \( f(0)<\eta \) and
		\[
		f(t) \leq f(0)+\int_{0}^{t} q(s) \rho_{\eta}(f(s)) \d s, \quad t \geq 0,
		\]
then
	\[
		f(t) \leq(f(0))^{\exp \left\{-\int_{0}^{t} q(s) \d s\right\}}.
		\]

	\end{lemma}

	Here we list some simple properties of \( \rho_{\eta}(\cdot) \) that will be used in the sequel. Their proofs can be found in Ren and Zhang \cite{REN2005Freidlin}, Cao and He \cite{Cao2003}, respectively.	
	\begin{lemma}\label{THM03}
		(1) \( \rho_{\eta} \) is decreasing in \( \eta \), i.e., \( \rho_{\eta_{1}} \leq \rho_{\eta_{2}} \)  if ~\( 1>\eta_{1}>\eta_{2} \).
		
		(2)  For any \( p>1 \), we have
		\[
		x^{p} \rho_{\eta}(x) \leq \rho_{\eta}(x^{1+p}) /(1+p).
		\]
	\end{lemma}
 \begin{lemma}\label{LEM:CONCAVE:01}
     Let \( \rho \) be a concave non-decreasing continuous function on \( \R_{+} \) such that \( \rho(0)=0 \). Then \( g(x):=\rho^{r}(x^{\frac{1}{s}}) \) is also a concave non-decreasing continuous function on \( \R_{+} \) such that \( g(0)=0 \) for all \( s \geqslant r \geqslant 1 \). Especially, $\rho_{0,\eta}(x):=\rho^2_{\eta}(x^{\frac{1}{2}})$ is a concave non-decreasing continuous function.
 \end{lemma}
	In what follows, we first introduce assumptions about the slow-fast motion  \eqref{EQ:SDE:01} and give the main theorem. To ensure the existence and uniqueness of  strong solutions to the system \eqref{EQ:SDE:01}, we first assume that for any $ z\in \R^{2d} $, $f_{2}(z)$ and $\sigma_{2}(z)$ are measurable and locally Lipschitz throughout this paper. Moreover, we need the following assumptions.
  
		\noindent \hypertarget{(H1)}{(H1)}. For any  $ x_1,x_2\in \R^d,z_1,z_2\in \R^{2d}, $ 
		\begin{equation*}
			\begin{aligned}
				|f_1(z_1)-f_1(z_2)| &\leq |z_1-z_2|g_1(|z_1-z_2|),\\
				\|\sigma_1(x_1)-\sigma_1(x_2)\|^2 &\leq |x_1-x_2|^2g_2(|x_1-x_2|),
			\end{aligned}
		\end{equation*}
		where $\|\cdot\|$ is the Hilbert-Schmidt norm of the matrix in $\R^{n\times m}$, and $ g_i$, $i=1,2$ are positive continuous functions, bounded on $ [1,+\infty)$ and satisfy
		\begin{equation}\label{eq:xiaohe}
\lim _{x \downarrow 0} \frac{g_{1}(x)}{\log x^{-1}}=0,	\quad  \lim _{x \downarrow 0} \frac{g_{2}(x)}{\log x^{-1}}=\lambda\in [0,\infty).
		\end{equation}


 
  
	\noindent \hypertarget{(H2)}{(H2)}.
     There exist positive constants \( \beta_{1}\)  and $C$ such that
     \[
		\begin{split}
			2\left\langle y_{1}-y_{2}, f_{2}\left(x_1, y_{1}\right)- f_{2}\left(x_2, y_{2}\right)\right\rangle & +\left\|\sigma_{2}\left(x, y_{1}\right)-\sigma_{2}\left(x, y_{2}\right)\right\|^{2} \\ \leq-\beta_{1}\left|y_{1}-y_{2}\right|^{2}+C|x_1-x_2|^2,
		\end{split}
     \]
for any $(x_1,y_1)$, $(x_2,y_2)\in\R^{2d}$.


    \noindent \hypertarget{(H3)}{(H3)}.
    There exists a positive constant $C$ such that
    \[
    \sup_{y\in \R^d}\|\sigma_2(x,y)\|^2\leq C(1+|x|^2),
    \]
    for any $(x,y)\in\R^{2d}$.
   \begin{remark}\label{RMK:01}
      By simple calculation, Assumption  (\hyperlink{(H1)}{H1}) implies that there is a \( \eta=\eta\left(p, \lambda\right)>0 \) sufficiently small such that
        \begin{equation}\label{eq:daxue}
        x^{p}\left(g_{1}(x)+g_{2}(x)\right) \leq C\left(d, p, \lambda\right)\rho_{\eta}\left(x^{p}\right), \quad x \in \mathbb{R}_{+}.
        \end{equation}  
 Furthermore, \( f_1\) and \(\sigma_1 \) satisfy the linearity growth condition, as shown in \cite[Corollary 3.4]{Cao2003}.
   \end{remark}
   
\begin{remark}
    By Assumption (\hyperlink{(H2)}{H2}) and direct calculations, we can deduce that 
    		\[
		\begin{split}
			2\left\langle y_{1}-y_{2}, f_{2}\left(x, y_{1}\right)- f_{2}\left(x, y_{2}\right)\right\rangle & +\left\|\sigma_{2}\left(x, y_{1}\right)-\sigma_{2}\left(x, y_{2}\right)\right\|^{2}  \leq-\beta_{1}\left|y_{1}-y_{2}\right|^{2},
		\end{split}
		\]
for any $y_1,y_2\in\R^d$.
Moreover, there exist two positive constants \( \beta_{2} \) and $\gamma $ such that\[
		\begin{split}
			2\left\langle y, f_{2}\left(x, y\right)\right\rangle  +\left\|\sigma_{2}\left(x, y\right)\right\|^{2}  \leq -\beta_{2}\left|y\right|^{2}+\gamma(1+|x|^{2}),
		\end{split}
		\]
for any $(x, y)\in\R^{2d}$.

\end{remark}
	\begin{lemma}
		Under assumptions  (\hyperlink{(H1)}{H1})-(\hyperlink{(H2)}{H2}), the stochastic slow-fast motion \eqref{EQ:SDE:01} has the unique solution $ (X_t^{(\varepsilon,\delta)}(x,y),Y_t^{(\varepsilon,\delta)}(x,y)) $ with initial value $(X_0^{(\varepsilon,\delta)},Y_0^{(\varepsilon,\delta)})= (x,y)$.   Moreover, when the fast variable's initial data is fixed as $ y_0, $ the mapping $ (t,x)\mapsto X_t^{(\varepsilon,\delta)}(x,y_0)$ is in $\C_T$, $\P\text{-a.s.}.$
	\end{lemma}
	
  We provide a precise statement of our main theorem and we shall give the proof of this theorem in Section \ref{SEC04}.
	\begin{theorem}\label{THM22}
		Suppose that (\hyperlink{(H1)}{H1})-(\hyperlink{(H3)}{H3}) hold. Then, the stochastic flow of the  slow variable \( X^{(\varepsilon, \delta)} \) in \eqref{EQ:SDE:01}   satisfies the LDP with rate function \( I: \C_T\rightarrow[0, \infty] \), i.e., for any $ A\in \mathcal{B}(\C_T), $
		\[
		-I\left(A^{\circ}\right) \leq \liminf_{\varepsilon \rightarrow 0} \varepsilon \log \nu_{{\varepsilon,\delta}}(A) \leq \limsup_{\varepsilon \rightarrow 0} \varepsilon \log \nu_{{\varepsilon,\delta}}(A) \leq-I(\bar{A}),
		\]
		where  $ \nu_{{\varepsilon,\delta}}$ denote   the distribution of  $ X_{\cdot}^{\varepsilon,\delta}(\cdot,y_0)$ in $ \C_T $,   \( A^{\circ} \) (or \( \bar{A} \))  denote interior (or closure) of \( A \) under the topology of $~\C_T$, and
		\begin{equation}\label{EQ10}
			I(f):=\frac{1}{2} \inf _{\{h \in \mathbb{H}: S(h)=f\}}\|h\|_{\mathbb{H}}^{2}, \quad f \in \mathbb{C}_{T} ,
		\end{equation}
		\( \inf (\varnothing)=+\infty \) by convention. Here, $S(h)$ is defined by \eqref{EQ:0629:01}.
	\end{theorem}
	
	\section{Auxiliary Lemmas}\label{SEC03}
	\subsection{The frozen process}
	For fixed  $x\in\R^d$, we introduce the frozen process associated to the fast variable as follows,
	\begin{equation}\label{EQ05}
		\d {Y}_{t}^x(y)=f_{2}\left({x}, {Y}_{t}^x(y)\right) \d t+\sigma_{2}\left({x}, {Y}_{t}^x(y)\right) \d W^2_{t},\quad {Y}_{0}^x(y) =y\in \R^d.
	\end{equation}
	 Since $f_{2}$ and $\sigma_{2}$ satisfy the  locally Lipschitz  condition,  it is easy to show that  
  \eqref{EQ05} has a unique local maximal solution which is also the unique global solution under the assumption (\hyperlink{(H2)}{H2}). Denote the unique strong solution by \( \left\{{Y}_{t}^{x}(y)\right\}_{t \geq 0} \),  which is a Markov process. Moreover,  (\hyperlink{(H2)}{H2}) ensures that there exists a unique invariant probability measure \( \mu_{{x}}\) for the frozen process, which can be deduced from Prato and Zabczyk \cite[Theorem 6.3.2]{prato1996}.  
	
	Let \( \left\{P_{t}^{x}\right\}_{t \geq 0} \) be the transition semigroup of \( \left\{Y_{t}^{x}(y)\right\}_{t \geq 0} \), that is, for any bounded measurable function \( f: \mathbb{R}^{d} \rightarrow \mathbb{R} \),
	\[
	P_{t}^{x} f(y):=\mathbb{E}\left[f\left(Y_{t}^{x}(y)\right)\right], ~ y \in \mathbb{R}^{d},~ t \geq 0 .
	\]
	Under the assumption (\hyperlink{(H2)}{H2}) , it is easy to prove that
	\begin{equation}\label{EQ24}
		\mathbb{E}[\left|Y_{t}^{x}(y)\right|^{2}]\leq \e^{-\beta_{2} t}|y|^{2}+\frac{\gamma}{\beta_{2}}(1+|x|^{2})
	\end{equation}
	and \( \left\{P_{t}^{x}\right\}_{t \geq 0} \) has a unique invariant measure \( \mu_{x}(\d y) \) satisfying
	\begin{equation}\label{EQ25}
		\int_{\mathbb{R}^{d}}|y|^{p} \mu_{x}(\d y) \leq C(1+|x|^{p})
	\end{equation}
	 for some \( p \geq 1 \) (see, e.g., \cite[Lemma 3.6, Proposition 3.8]{LIU2020}).
	
	In what follows, we give some estimations about the frozen process \eqref{EQ05}.
	\begin{lemma}\label{THM04}
		Suppose that  (\hyperlink{(H2)}{H2}) holds. For fixed \( x\in \R^d \), we have
		\[
		\mathbb{E}[\left|Y_{t}^{x}(y_1)-Y_{t}^{x}(y_2)\right|^{2}] \leq \left|y_{1}-y_{2}\right|^{2}\e^{-\beta_{1} t},  
		\]
  for any $t\geq 0$, $y_{1}, y_{2} \in \mathbb{R}^{d}$.
	\end{lemma}
	\begin{proof}
		By \ito's formula and  (\hyperlink{(H2)}{H2}), we obtain
		\begin{equation*}
			\begin{split}
				 \E & [\e^{\beta_{1}t}|Y_t^x(y_1)-Y_t^x(y_2)|^2] \\
				& = |y_1-y_2|^2+\E\int_{0}^{t}2\langle Y_t^x(y_1)-Y_t^x(y_2), f_2(x,Y_t^x(y_1))-f_2(x,Y_t^x(y_2))	\rangle\d s\\
				& \quad+ \E\int_{0}^{t} \|\sigma_{2}(x,Y_t^x(y_1)-\sigma_{2}(x,Y_t^x(y_2)))\|^2 \d s
				+\beta_{1}\E\int_{0}^{t}\e^{\beta_{1}s}|Y_t^{x}(y_1) -Y_t^{x}(y_2) |^2\d s\\
				&\leq   |y_1-y_2|^2.
			\end{split}
		\end{equation*}
	Hence,
	\[ 
	\E[|Y_t^{x}(y_1) -Y_t^x(y_2)|^2] \leq  |y_1-y_2|^2\e^{-\beta_{1}t}.
	 \]
	 The proof is complete.
	\end{proof}
	\begin{lemma}\label{THM05}
		Suppose that (\hyperlink{(H1)}{H1}) and  (\hyperlink{(H2)}{H2}) hold. For fixed \( x \in \mathbb{R}^{d} \) , there exist two positive constants  \( C\) and $ \eta_1 $ small enough, such that 
	\begin{equation}\label{EQ08}
		\bigg|\mathbb{E}\left[f_{1}\left( x, Y_{t}^{x}(y)\right)\right]-\int_{\mathbb{R}^{d}} f_{1}( x, u) \mu_{x}(\d u)\bigg| \leq C\rho_{\eta_1}\left( \e^{-\beta_{1}t/2}(1+|x|+|y|)\right),
	\end{equation}
  for any $t\geq 0$, $y\in \mathbb{R}^{d}$.
	\end{lemma}
	\begin{proof}
  By Remark \ref{RMK:01}, there exists two positive constants $C$ and  $\eta_{1}$ such that for any $x \in \mathbb{R}_{+}$,  $xg_1(x)\leq C\rho_{\eta_{1}}(x)$. Notice that $ \rho_{\eta_1} $ is a concave function. By  (\hyperlink{(H1)}{H1}),  Jensen's inequality and Lemma \ref{THM04}, we have that 
	\begin{align*}
			\bigg|\E&[f_1(x,Y_t^x(y))]- \int_{\R^d}f_1(x,u)\mu_{{x}}(\d u)\bigg| \\
   & \leq \int_{\R^d} \E [\left|f_1(x,Y_t^{x}(y)) - f_1(x,Y_t^x(u))\right|] \mu_x(\d u)\\
			&\leq  \int_{\R^d} \E \left[|Y_t^{x}(y)-Y_t^x(u)|g_1(|Y_t^{x}(y)-Y_t^x(u)|)\right]\mu_x(\d u)\\
			& \leq C\int_{\R^d} \E [\rho_{\eta_1}(|Y_t^x(y)-Y_t^x(u)|)]\mu_x(\d u)\\
			& \leq C\rho_{\eta_1}\left(\int_{\R^d}  \E [|Y_t^x(y)-Y_t^x(u)|]\mu_x(\d u)\right)\\
			& \leq  C\rho_{\eta_1}\left( \int_{\R^d}\e^{-\beta_1{t}/2}|y-u|\mu_{x}(\d u)  \right)\\
			& \leq  C\rho_{\eta_1}\left( \e^{-\beta_{1}t/2}(1+|x|+|y|)\right),
	\end{align*}
 where  \eqref{EQ25} has been used in the last inequality.  The proof is complete.
	\end{proof}

	\begin{lemma}\label{THM06}
		Suppose that  (\hyperlink{(H2)}{H2}) hold. For fixed \( y \in  \mathbb{R}^{d} \), we have
		\[
		\mathbb{E}[|Y_{t}^{x_{1}}(y)-Y_{t}^{x_{2}}(y)|^{2}] \leq C|x_{1}-x_{2}|^{2},
		\]
  for any $t\geq 0$ and $x_{1}, x_{2}\in \R^d$.
	\end{lemma}
	\begin{proof} 
	Recall the fundamental inequality
	\begin{equation}\label{eq:yao}
		2ab\leq \epsilon a^2+\frac{1}{\epsilon}b^2
	\end{equation}
	for all $a$, $b\in\R$ and $\epsilon>0$. By \ito's  formula,  (\hyperlink{(H2)}{H2}) and  (\hyperlink{(H4)}{H4}),  we obtain
	\begin{align*}
		\begin{split}
			 \E &\left[\e^{\beta_{1}t/2}|Y_t^{x_1}(y)-Y_t^{x_2}(y)|^2\right] \\
			& = \frac{\beta_{1}}{2}\E\int_{0}^{t}\e^{\beta_{1}s/2}|Y_s^{x_1}(y)-Y_{s}^{x_2}(y)|^2\d s+ \E\int_{0}^{t} \e^{\beta_{1}s/2}\left( 2\langle Y_s^{x_1}(y)-Y_{s}^{x_2}(y),f_2(x_1,Y_s^{x_1}(y))\right.
   \\ 
   &\quad -f_2(x_2,Y_s^{x_2}(y))\rangle +\|\sigma_{2}(x_1,Y_s^{x_1}(y))-\sigma_{2}(x_2,Y_s^{x_2}(y))\|^2 	)\d s\\
			&  \leq  
			-\frac{\beta_{1}}{2}\E\int_{0}^{t}\e^{\beta_{1}s/2}|Y_s^{x_1}(y)-Y_{s}^{x_2}(y)|^2\d s+\frac{2\alpha_1}{\beta_{1}}|x_1-x_2|^2\e^{\beta_{1}t/2}\\
			&\quad+2\E\int_{0}^{t} \e^{\beta_{1}s/2}\langle Y_s^{x_1}(y)-Y_{s}^{x_2}(y),f_2(x_1,Y_s^{x_2}(y))-f_2(x_2,Y_s^{x_2}(y))\rangle\d s\\
			& \quad +2\E\int_{0}^{t} \e^{\beta_{1}s/2}\langle \sigma_{2}(x_1,Y_s^{x_1}(y))-\sigma_{2}(x_1,Y_s^{x_2}(y)),\sigma_{2}(x_1,Y_s^{x_2}(y)) -\sigma_{2}(x_2,Y_s^{x_2}(y))\rangle_{HS} \d s\\
			& \leq\frac{2}{\beta_{1}}\left(\alpha_1+\frac{\alpha_1}{\epsilon_1}+\frac{1}{\epsilon_2}\right)|x_1-x_2|^2\e^{\beta_{1}t/2}\\
               & \quad + \left(\epsilon_1+\alpha_1\epsilon_2-\frac{\beta_{1}}{2}\right)\E\int_{0}^{t}\e^{\beta_{1}s/2}|Y_s^{x_1}(y)-Y_{s}^{x_2}(y)|^2\d s.
		\end{split}
	\end{align*}
	Choosing $\epsilon_1=\beta_{1}/4$ and $\epsilon_2=\beta_{1}/4\alpha_1$ yields that $ \epsilon_1+\alpha_1\epsilon_2-\frac{\beta_{1}}{2} =0, $ so 
	\begin{equation*}\label{eq:shuijiao}
		\mathbb{E}[|Y_{t}^{x_{1}}(y)-Y_{t}^{x_{2}}(y)|^{2}] \leq C\left|x_{1}-x_{2}\right|^{2},
	\end{equation*}
	 where $C$ is a positive constant depending on $\alpha_1$ and $\beta_1$. The proof is complete.
	\end{proof}
	\subsection{The skeleton process}
	Let $P_1,P_2:\R^{2d}\to \R^d$ be projection operators on $\R^{2d}$. For any $h\in \mathbb{H},$ set $h_1:=P_1h,~h_2:=P_2h,$ then $h=(h_1,h_2)$.
 Now we introduce the skeleton process of the slow variable in  \eqref{EQ:SDE:01}, which  is defined by the following ODE,
	\begin{equation}\label{EQ06}
		\d {X}_{t}^{h}(x)=\bar{f}_{1}({X}_{t}^h(x)) \d t+\sigma_{1}({X}_{t}^h(x)) \dot{h}_1(t)\d t,\quad X_0^h(x)=x\in \R^d.
	\end{equation}
	Here, \( h\in \mathbb{H} \) is called the corresponding control function, and $ \bar{f}_{1} $ is the averaging coefficient of $ f_1 $, i.e.,
	\begin{equation}\label{EQ07}
		\bar{f}_{1}(a)=\int_{\mathbb{R}^{{d}}} f_{1}(a, {y})~\mu_{{a}}(\d {y}).
	\end{equation}
 According to the above Lemmas \ref{THM04}-\ref{THM06}, we can deduce the following result.
	\begin{theorem}\label{THM07}
		Suppose that (\hyperlink{(H1)}{H1})-(\hyperlink{(H3)}{H3}) hold, then ODE \eqref{EQ06} has a unique solution,
	\end{theorem}
	\begin{proof}
First, we give the estimation of the averaging coefficient $ \bar{f}_1.$	
Set 
\begin{equation}\label{EQ53}
	\gamma(a) := \sup_{0<|x_1-x_2|\leq a} \frac{|\bar{f}_1(x_1)-\bar{f}_1(x_2)|}{|x_1-x_2|}
\end{equation}
for any $a>0$ and $\gamma(0)=0$. By \eqref{eq:xiaohe}, for any $ \epsilon>0, $ there exists a positive constant $ \theta $ such that $ g_1(x)\leq \epsilon\log x^{-1}$ for any $ 0< x\leq2\theta$. 
For any $ x_1,x_2\in \R^d$ and initial value $ y\in \R^d, $  set $ D:= \{\sup_{t\geq 0}|Y_t^{x_1}(y)-Y_t^{x_2}(y)|\leq \theta\}.$ Profit from Lemmas \ref{THM05}-\ref{THM06},  we have that
\begin{align*}
	 |\bar{f}_1&(x_1) -\bar{f}_1(x_2)| \\
 &\leq \left| \int_{\R^d} {f}_1(x_1,u)\mu_{x_1}(\d u) -\E [f_1(x_1,Y_t^{x_1}(y))]\right| \\
 &\quad+\left| \int_{\R^d} {f}_1(x_2,u)\mu_{x_2}(\d u) -\E [f_1(x_2,Y_t^{x_2}(y))]\right|\\
	& \quad +|\E\left[ f_1(x_1,Y_t^{x_1}(y))-f_1(x_2,Y_t^{x_2}(y))\right]\cdot (\1_D+\1_{D^c})|\\
	&  \leq C \rho_{\eta_1}\left( \e^{-\beta_{1}t/2}(1+|x_1|+|y|)\right) +C\rho_{\eta_1}\left( \e^{-\beta_{1}t/2}(1+|x_2|+|y|)\right) \\
	& \quad +  \E [\sqrt{|x_1-x_2|^2+|Y_t^{x_1}(y)-Y_t^{x_2}(y)|^2} \cdot g_1(\sqrt{|x_1-x_2|^2+|Y_t^{x_1}(y)-Y_t^{x_2}(y)|^2}) \\
        &\qquad\quad  \cdot (\1_{D}+\1_{D^c})]\\
	& \leq C \rho_{\eta_1}\left( \e^{-\beta_{1}t/2}(1+|x_1|+|y|)\right) +C\rho_{\eta_1}\left( \e^{-\beta_{1}t/2}(1+|x_2|+|y|)\right) \\
	&  \quad +\epsilon\E\left[\rho_{\eta_1} \left( \sqrt{|x_1-x_2|^2+|Y_t^{x_1}(y)-Y_t^{x_2}(y)|^2}\right)\right]\\
        &\quad+C\E [\sqrt{|x_1-x_2|^2+|Y_t^{x_1}(y)-Y_t^{x_2}(y)|^2}] \\
	& \leq C \rho_{\eta_1}\left( \e^{-\beta_{1}t/2}(1+|x_1|+|y|)\right) +\rho_{\eta_1}\left(\e^{-\beta_{1}t/2}(1+|x_2|+|y|)\right) \\
	&  \quad +\epsilon\rho_{\eta_1} \left( \E \sqrt{|x_1-x_2|^2+|Y_t^{x_1}(y)-Y_t^{x_2}(y)|^2}\right)+C(|x_1-x_2|+\E|Y_t^{x_1}(y)-Y_t^{x_2}(y)|)\\
	& \leq C\rho_{\eta_1}\left( \e^{-\beta_{1}t/2}(1+|x_1|+|y|)\right) +C\rho_{\eta_1}\left( \e^{-\beta_{1}t/2}(1+|x_2|+|y|)\right)  +C\epsilon\rho_{\eta_1} ( |x_1-x_2|).
\end{align*}
Notice that $ \lim_{x\downarrow 0} \rho_{\eta_1}(x) =0$, let $ t\to\infty$, 
\begin{equation}\label{EQ67}
	|\bar{f}_1(x_1) -\bar{f}_1(x_2)| \leq 2\epsilon\rho_{\eta_1} ( |x_1-x_2|).
\end{equation}
Since $ \gamma(\cdot) $ is defined by \eqref{EQ53}, so indeed we have that
\begin{equation}\label{EQ27}
	|\bar{f}_1(x_1) -\bar{f}_1(x_2)| \leq |x_1-x_2|\gamma(|x_1-x_2|)\leq 2\epsilon\rho_{\eta_1} ( |x_1-x_2|).
\end{equation}
Due to the linear growth of  $\rho_{\eta_1}(\cdot)$ on $ [1,+\infty),$ we can see that  $ \gamma(\cdot)$ is also a  positive 
function, {bounded on $ [1,+\infty)$}   and satisfing
\[ 
\lim _{x \downarrow 0} \frac{\gamma(x)}{\log x^{-1}}= 0.
\]
	 
	 Next, for any given $T>0$, we use the Euler scheme to prove that the ODE \eqref{EQ06} has a unique solution denoted by $\{X_t^h(x)\}_{t\in [0,T]}$. Let \( X_{\cdot}^{n} \) be the solution of the following recursive system:
	 \[
	 X_{t}^{n}=x+\int_{0}^{t} \bar{f}_1(X_{s_{n}}^{n}) \d s+\int_{0}^{t}\sigma_1(X_{s_{n}}^{n})\dot{h}_1(s) \d s,
	 \]
	 where \( s_{n}=\left[s 2^{n}\right] / 2^{n} \).
	 It is not hard to find that there is a constant \( C \) such that for any \( t \in[0, T], \)
	\begin{equation}\label{EQ26}
		 |X_{t}^{n}-X_{t_{n}}^{n}| \leq C 2^{-n}.
	\end{equation}
	 Set
	 \[
	 \varphi_{n}(t):=|X_{t}^{n+1}-X_{t}^{n}|.
	 \]
	 It follows from \eqref{eq:daxue} that there exist two positive constants $C$ and  $\eta_{2}$ such that $xg_2(x)\leq C\rho_{\eta_{2}}(x)$ for any $x \in \mathbb{R}_{+}$. Now by \eqref{EQ27} and \eqref{EQ26} we have
	 \begin{align*}
	 	\varphi_{n}(t) & \leq\left|\int_{0}^{t}(\sigma_1(X_{s_{n+1}}^{n+1})-\sigma_1(X_{s_{n}}^{n}))\dot{h}_1(s) \d s\right|+\left|\int_{0}^{t} \bar{f}_1(X_{s_{n+1}}^{n+1})-\bar{f}_1(X_{s_{n}}^{n}) \d s\right| \\
	 	& \leq 2\epsilon \rho_{\eta_1}(C2^{-n})+\left|\int_{0}^{t}(\sigma_1(X_{s}^{n+1})-\sigma_1(X_{s}^{n})) \dot{h}_1(s)\d s\right|+\left|\int_{0}^{t} \bar{f}_1(X_{s}^{n+1})-\bar{f}_1(X_{s}^{n}) \d s\right| \\
	 	& \leq 2\epsilon \rho_{\eta_1}(C2^{-n})+\int_{0}^{t} \sqrt{C \varphi_{n}(s) \rho_{\eta_2}(\varphi_{n}(s))} \cdot|\dot{h}_1(s)| \d s+\int_{0}^{t} 2\epsilon\rho_{\eta_1}( \varphi_{n}(s)) \d s \\
	 	& \leq 2\epsilon \rho_{\eta_1}(2^{-n})+C\int_{0}^{t} \rho_{\eta^{\prime}}(\varphi_{n}(s)) \cdot (1+|\dot{h}_1(s)|) \d s
	 \end{align*}
	 where \( 0<\eta^{\prime}<\eta_1\wedge\eta_2 \) is small enough. By Lemma \ref{THM02}-(2), we obtain that 
	 \[
	 \begin{aligned}
	 	C\varphi_{n}(t) &\leq[2C\epsilon \rho_{\eta_1}(C2^{-n})]^{\exp \left\{-C \int_{0}^{t}(1+|\dot{h}_1(s)|) \d s\right\}} \\
	 	             & \leq[2C\epsilon \rho_{\eta_1}(C2^{-n})]^{\exp \left\{-C T\left(1+\|h\|_{\mathbb{H}}\right)\right\}} \\
	 	             & \leq C 2^{-\alpha n},
	 \end{aligned}
	 \]
	hold for \( n \) sufficiently large and some \( \alpha>0 \).
	 Hence \( X_{n}(t) \) converges uniformly on \( [0, T] \) and taking the limit we obtain the existence. The uniqueness is deduced from a similar estimate. The proof is complete.
	\end{proof}

	\subsection{The controlled system}
	For any $ h^{(\varepsilon,\delta)}=(P_1h^{(\varepsilon,\delta)},P_2h^{(\varepsilon,\delta)})=(h_1^{(\varepsilon,\delta)},h_2^{(\varepsilon,\delta)}) \in \mathcal{A},$ the following is called the controlled system corresponding  \eqref{EQ:SDE:01},
	\begin{equation}\label{EQ12}
		\left\{\begin{aligned}
			\d \hat{X}_{t}^{(\varepsilon, \delta)}&=  f_{1}(\hat{X}_{t}^{(\varepsilon, \delta)}, \hat{Y}_{t}^{(\varepsilon, \delta)}) \d t+\sigma_{1}(\hat{X}_{t}^{(\varepsilon, \delta)}) \dot{h}_1^{(\varepsilon, \delta)}(t) \d t+\sqrt{\varepsilon} \sigma_{1}(\hat{X}_{t}^{(\varepsilon, \delta)}) \d W_{t}^1, \\
			\delta \d \hat{Y}_{t}^{(\varepsilon, \delta)}& =  f_{2}(\hat{X}_{t}^{(\varepsilon, \delta)}, \hat{Y}_{t}^{(\varepsilon, \delta)}) \d t+\sqrt{\frac{\delta}{\varepsilon}} \sigma_{2}(\hat{X}_{t}^{(\varepsilon, \delta)}, \hat{Y}_{t}^{(\varepsilon, \delta)}) \dot{h}_2^{(\varepsilon, \delta)}(t) \d t+\sqrt{\delta} \sigma_{2}(\hat{X}_{t}^{(\varepsilon, \delta)}, \hat{Y}_{t}^{(\varepsilon, \delta)}) \d W_{t}^2.
		\end{aligned}\right.
	\end{equation}	
 Following the same argument with Ren and Zhang \cite[Theorem 3.6]{REN2005}, one can deduce the next result.
	\begin{lemma}\label{THM120}
		Suppose that   (\hyperlink{(H1)}{H1})-(\hyperlink{(H3)}{H3}) hold, then the system \eqref{EQ12}~ has a unique strong solution $ \{(\hat{X}_t^{(\varepsilon,\delta)},\hat{Y}_t^{(\varepsilon,\delta)})\}_{t\geq 0}$ with the initial value $(\hat{X}_0^{(\varepsilon,\delta)},\hat{Y}_0^{(\varepsilon,\delta)})=(x,y)$.
	\end{lemma}
	
	Hereafter, given a family \( \{h^{(\varepsilon,\delta)}, \varepsilon>\delta>0\} \subset \mathcal{A}_{N} \), let \( (\hat{X}_{\cdot}^{(\varepsilon, \delta)},\hat{Y}_{\cdot}^{(\varepsilon, \delta)}) \) be the corresponding solution to the system \eqref{EQ12}.
Moreover, we list  some important prior estimations for the controlled system $ (\hat{X}_{\cdot}^{(\varepsilon,\delta)},\hat{Y}_{\cdot}^{(\varepsilon,\delta)})  $, which are core estimations to derive the LDP for stochastic flow by the method of weak convergence. 
	\begin{lemma}\label{THM19}
	Suppose that (\hyperlink{(H1)}{H1})-(\hyperlink{(H3)}{H3}) hold. Let \( N\geq 1 \) , then there exists a positive constant \( C=C(N,T) \) such that for every \( h^{(\varepsilon, \delta)}\in \mathcal{A}_{N} \), we have
		\[
		\mathbb{E}[\sup_{t\in [0,T]}|\hat{X}_t^{(\varepsilon, \delta)}|^{2}] \leq C(1+|x|^2+|y|^2),
		\]
  and 
  \[
		\sup_{t\in [0,T]}\mathbb{E}[|\hat{Y}_t^{(\varepsilon, \delta)}|^{2}] \leq C(1+|x|^2+|y|^2).
  \]
	\end{lemma}

\begin{proof}
    Recall that \( f_1\) and \(\sigma_1 \) satisfy the linearity growth condition. By the $C_r$-inequality, BDG's inequality and the definition  of $\mathcal{A}_{N}$, we have
       \begin{align*}
        \E[\sup_{t \in[0, T]}|\hat{X}_{t}^{(\varepsilon, \delta)}|^{2}]& \leq C|x|^{2}+\E\left[\sup _{t \in[0, T]}\left|\int_{0}^{t} f_{1}(\hat{X}_{s}^{(\varepsilon, \delta)}, \hat{Y}_{s}^{(\varepsilon, \delta)}) \d s\right|^{2}\right] \\
        &\quad+C \E\left[\sup _{t \in[0, T]}\left|\int_{0}^{t} \sigma_{1}(\hat{X}_{s}^{(\varepsilon, \delta)}) \dot{h}_1^{(\varepsilon, \delta)}(s) \d s\right|^{2}\right] \\
        &\quad+C \varepsilon \E\left[\sup _{t \in[0, T]}\left|\int_{0}^{t} \sigma_{1}(\hat{X}_{s}^{(\varepsilon, \delta)}) \d W_{s}^1\right|^{2}\right] \\
        &\leq C|x|^{2}+C \E\left[\int_{0}^{T}\left|f_{1}(\hat{X}_{t}^{(\varepsilon, \delta)}, \hat{Y}_{t}^{(\varepsilon, \delta)})\right|^{2} \d t\right] \\
        &\quad+C \E\left[\int_{0}^{T}\|\sigma_{1}(\hat{X}_{t}^{(\varepsilon, \delta)})\|^{2} \d t \cdot \int_{0}^{T}|\dot{h}_1^{(\varepsilon, \delta)}(t)|^{2} \d t\right] \\
        &\quad+C \varepsilon \E\left[\int_{0}^{T}\|\sigma_{1}(\hat{X}_{t}^{(\varepsilon, \delta)})\|^{2} \d t\right] \\
        &\leq C|x|^{2}+C \E\left[\int_{0}^{T}|\hat{X}_{t}^{(\varepsilon . \delta)}|^{2} \d t\right]+C \E\left[\int_{0}^{T}|\hat{Y}_{t}^{(\varepsilon,\delta)}|^{2} \d t\right] \\
        &\leq C|x|^{2}+C \int_{0}^{T} \E[\sup _{s \in[0, t]}|\hat{X}_{s}|^{2}] \d t +C \E\left[\int_{0}^{T}|\hat{Y}_{t}^{(\varepsilon,\delta)}|^{2} \d t\right]. \\
       \end{align*}
   Applying \ito's formula to $|\hat{Y}_{t}^{(\varepsilon, \delta)}|^{2}$, it follows from (\hyperlink{(H3)}{H3}) that for any $\epsilon>0$,
   \begin{align*}
    \E[|\hat{Y}_{t}^{(\varepsilon, \delta)}|^{2}]&=  |y|^{2}+\frac{1}{\delta} \E\left[\int_{0}^{t} 2\langle\hat{Y}_{s}^{(\varepsilon, \delta)}, f_{2}(\hat{X}_{s}^{(\varepsilon, \delta)}, \hat{Y}_{s}^{(\varepsilon, \delta)})\rangle +\|\sigma_{2}(\hat{X}_{s}^{(\varepsilon, \delta)}, \hat{Y}_{s}^{(\varepsilon, \delta)})\|^{2} \d s\right] \\
    & \quad+\sqrt{\frac{1}{\varepsilon \delta}} \E\left[\int_{0}^{t} 2\langle\hat{Y}_{s}^{(\varepsilon, \delta)}, \sigma_{2}(\hat{X}_{s}^{(\varepsilon, \delta)}, \hat{Y}_{s}^{(\varepsilon, \delta)}) \dot{h}_2^{(\varepsilon, \delta)}(s)\rangle\d s\right]. \\
    &\leq  |y|^{2}+\left(\sqrt{\frac{1}{\varepsilon \delta}} \epsilon-\frac{\beta_{2}}{\delta}\right) \E\left[\int_{0}^{t}|\hat{Y}_{s}^{(\varepsilon, \delta)}|^{2} \d s\right] +\frac{\gamma}{\delta} \E\left[\int_{0}^{t}(1+|\hat{X}_{s}^{(\varepsilon, \delta)}|^{2}) \d s\right] \\
    &\quad +\frac{1}{\epsilon} \sqrt{\frac{1}{\varepsilon \delta}} \E\left[\int_{0}^{t}\|\sigma_{2}(\hat{X}_{s}^{(\varepsilon, \delta)}, \hat{Y}_{s}^{(\varepsilon, \delta)})\|^{2}\cdot|\dot{h}_2^{(\varepsilon, \delta)}(s)|^{2} \d s\right] \\
    & \leq  C(1+|y|^{2})+\left(\sqrt{\frac{1}{\varepsilon \delta}} \epsilon-\frac{\beta_{2}}{\delta}\right) \E\left[\int_{0}^{t}|\hat{Y}_{s}^{(\varepsilon, \delta)}|^{2} \d s\right]  +C \E[\sup _{s \in[0, t]}|\hat{X}_{s}^{(\varepsilon, \delta)}|^{2}],
\end{align*}
  where (\hyperlink{(H4)}{H4}) has been used in the last inequality. Letting ${\epsilon}$ be sufficiently small such that ${\sqrt{\frac{1}{\varepsilon \delta}} \epsilon-\frac{\beta_{2}}{\delta} \leq 0}$, we have
\[
\E[|\hat{Y}_{t}^{(\varepsilon, \delta)}|^{2}] \leqslant C(1+|y|^{2})+C \E[\sup _{s \in[0, t]}|\hat{X}_{s}^{(\varepsilon, \delta)}|^{2}] ,
\]
which implies
\[
\E[\sup _{t \in[0, T]}|\hat{X}_{t}^{(\varepsilon, \delta)}|^{2}] \leq C(1+|x|^{2}+|y|^{2})+C \int_{0}^{T} \E[\sup _{s \in[0, t]}|\hat{X}_{s}^{(\varepsilon, \delta)}|^{2}] \d t.
\]
By \gron's inequality,
\[
\E[\sup _{t \in[0, T]}|\hat{X}_{t}^{(\varepsilon, \delta)}|^{2}] \leqslant C(1+|x|^{2}+|y|^{2}).
\]
Moreover,
\[
\E[|\hat{Y}_{t}^{(\varepsilon, \delta)}|^{2}] \leqslant C(1+|x|^2+|y|^{2}),
\]
for any $t\geq 0$. The proof is complete.
\end{proof}
\begin{lemma}\label{THM17}
	 Suppose that (\hyperlink{(H1)}{H1})-(\hyperlink{(H4)}{H4}) hold.  Let \( N\geq 1 \) and \( p \geq 2 \), then there is a positive constant \( C=C(N,T,x,y)\) independent of \( ~\varepsilon \) such that
		\[
		\mathbb{E}[|\hat{X}_{t}^{(\varepsilon,\delta)}-\hat{X}_{s}^{(\varepsilon,\delta)}|^{2}] \leq C|t-s|,
		\]
		for all \( s\leq t \in[0, T] \).
	\end{lemma}
	\begin{proof}
		The proof is a standard argument.  Note that 
		\[ 
		\begin{split}
			\hat{X}_{t}^{(\varepsilon, \delta)}-\hat{X}_{s}^{(\varepsilon, \delta)} & =   \int_{s}^{t}f_{1}(\hat{X}_{r}^{(\varepsilon, \delta)}, \hat{Y}_{r}^{(\varepsilon, \delta)}) \d r+\int_{s}^{t}\sigma_{1}(\hat{X}_{r}^{(\varepsilon, \delta)}) \dot{h}_1^{(\varepsilon,\delta)}(r) \d r+\sqrt{\varepsilon} \int_{s}^{t}\sigma_{1}(\hat{X}_{t}^{(\varepsilon, \delta)}) \d W_{r}^1.
		\end{split}
		\]
		By  the $ C_r $-inequality, we have
		\begin{equation*}
			\begin{split}
				\E [|	\hat{X}_{t}^{(\varepsilon, \delta)}-\hat{X}_{s}^{(\varepsilon, \delta)}|^2]  
				& \leq 3\E\left[\left| \int_{s}^{t}f_{1}(\hat{X}_{r}^{(\varepsilon, \delta)}, \hat{Y}_{r}^{(\varepsilon, \delta)}) \d r\right|^2+ \left|\int_{s}^{t}\sigma_{1}(\hat{X}_{r}^{(\varepsilon, \delta)}) \dot{h}_1^{(\varepsilon,\delta)}(r) \d r\right|^2 \right. \\
				& \qquad \qquad  + \left.\left|\sqrt{\varepsilon}\int_{s}^{t} \sigma_{1}(\hat{X}_{r}^{(\varepsilon, \delta)}) \d W_{r}^1 \right|^2\right]=:\sum_{i=1}^{3}I_i.
			\end{split}
		\end{equation*}
		For $ I_1 ,$  by \hold's  inequality, the linear growth of $f_1$ and Lemma \ref{THM19}, we have that 
		\[ 
            \begin{aligned}
                I_1&\leq 3|t-s|\mathbb{E}\int_{s}^{t}|f_{1}(\hat{X}_{r}^{(\varepsilon, \delta)}, \hat{Y}_{r}^{(\varepsilon, \delta)})|^p \d r\\
                &\leq  6|t-s|\int_{s}^{t}(1+\E[|\hat{X}_{r}^{(\varepsilon, \delta)}|^2]+\E[|\hat{Y}_{r}^{(\varepsilon, \delta)}|^2])\d r\\
                &\leq C|t-s|^{2}.
            \end{aligned}
		\]
       Similarly, for $ I_2,$  the definition of $ \mathcal{A}_N$ implies that
		\[ 
            \begin{aligned}
                I_2 &\leq 3CN^2\E\left[\int_{s}^{t}(1+|\hat{X}_{r}^{(\varepsilon, \delta)}|^2 )\d r\right]\\
                &\leq C\int_{s}^{t}(1+\E[|\hat{X}_{r}^{(\varepsilon, \delta)}|^2])\d r\\
                &\leq C|t-s|.
            \end{aligned}
		\]
	For $ I_3 $,  by \ito's isometry and   (\hyperlink{(H1)}{H1}), we obtain
		\[ 
		\begin{split}
			I_3 & \leq \varepsilon  \int_{s}^{t} \E[\|  \sigma_{1}(\hat{X}_{r}^{(\varepsilon, \delta)})\|^2]\d r
			\leq C|t-s|.
		\end{split}
		\]
		To sum up, we obtain the desired assertion.
	\end{proof}
 
        In the sequel, we sometimes emphasize the system  \( (\hat{X}_{\cdot}^{(\varepsilon, \delta)},\hat{Y}_{\cdot}^{(\varepsilon, \delta)} )\) with the initial value $(\hat{X}_0^{(\varepsilon,\delta)},\hat{Y}_0^{(\varepsilon,\delta)})=(x,y)$ by \( (\hat{X}_{\cdot}^{(\varepsilon, \delta)}(x),\hat{Y}_{\cdot}^{(\varepsilon, \delta)}(y)). \)   Now	we can prove  the following key moment estimate of the slow variable $\hat{X}_{\cdot}^{(\varepsilon, \delta)}(x)$ starting from different initial points. To this aim, given a family \( \{h^{(\varepsilon,\delta)}, \varepsilon>\delta>0\} \subset \mathcal{A}_{N} \), let the auxiliary process $ (\tilde{X}_{\cdot}^{(\varepsilon, \delta)},\tilde{Y}_{\cdot}^{(\varepsilon, \delta)}) $ be the solution to the following SDE,
	\begin{equation}\label{EQ35}
		\begin{cases}
			\d \tilde{X}_{t}^{(\varepsilon, \delta)}=f_{1}(\hat{X}_{t(\Delta)}^{(\varepsilon, \delta)}, \tilde{Y}_{t}^{(\varepsilon, \delta)}) \d t+\sigma_{1}(\tilde{X}_{t}^{(\varepsilon, \delta)}) \dot{h}_1^{(\varepsilon, \delta)}(t) \d t ,\quad \tilde{X}_{0}^{(\varepsilon, \delta)}=x,\\
			\d \tilde{Y}_{t}^{(\varepsilon, \delta)}=\frac{1}{\delta} f_{2}(\hat{X}_{t(\Delta)}^{(\varepsilon, \delta)}, \tilde{Y}_{t}^{(\varepsilon, \delta)}) \d t+\frac{1}{\sqrt{\delta}} \sigma_{2}(\hat{X}_{t(\Delta)}^{(\varepsilon, \delta)}, \tilde{Y}_{t}^{(\varepsilon, \delta)}) \d W_{t}^2,  \quad \tilde{Y}_{0}^{(\varepsilon, \delta)}=y,
		\end{cases}
	\end{equation}
	where $ t(\Delta):=\left[\frac{t}{\Delta}\right] \Delta. $ 
 Hereafter, we propose several lemmas which will be used to prove the key moment estimate.
  \begin{lemma}\label{LEM:0729:01}
Suppose that  (\hyperlink{(H1)}{H1})-(\hyperlink{(H3)}{H3}) hold.  For $x,y\in \R^d, N\geq 1,$ there exists a constant $C=C(x,y,T,N)>0$ such that for any $\varepsilon,\delta > 0$ small enough,
\[
\E \int_{0}^{T}|\hat{Y}_t^{(\varepsilon,\delta)}-\tilde{Y}_t^{(\varepsilon,\delta)}|^2\d t\leq C(\Delta+\delta/\varepsilon).
\]
\begin{proof}
    By \eqref{EQ12} and \eqref{EQ35},
    \begin{equation*}
        \begin{split}
       \hat{Y}_t^{(\varepsilon,\delta)}-\tilde{Y}_t^{(\varepsilon,\delta)} &= \frac{1}{\delta}\int_{0}^{t} [f_{2}(\hat{X}_{s}^{(\varepsilon, \delta)}, \hat{Y}_{s}^{(\varepsilon, \delta)})-f_{2}(\hat{X}_{s(\Delta)}^{(\varepsilon, \delta)}, \tilde{Y}_{s}^{(\varepsilon, \delta)})]\d s \\
       & \quad +\frac{1}{\sqrt{\delta\varepsilon}}\int_{0}^{t} \sigma_{2}(\hat{X}_{s}^{(\varepsilon, \delta)}, \hat{Y}_{s}^{(\varepsilon, \delta)}) \dot{h}_2^{(\varepsilon, \delta)}(s) \d s\\
       &\quad +\frac{1}{\sqrt{\delta}}\int_{0}^{t}[\sigma_{2}(\hat{X}_{s}^{(\varepsilon, \delta)}, \hat{Y}_{s}^{(\varepsilon, \delta)})-\sigma_{2}(\hat{X}_{s(\Delta)}^{(\varepsilon, \delta)}, \tilde{Y}_{s}^{(\varepsilon, \delta)})] \d W_{s}^2.
        \end{split}
    \end{equation*}
    Applying \ito's formula yields that
    \begin{align}\label{EQ:0729_01}
            |\hat{Y}_t^{(\varepsilon,\delta)}-\tilde{Y}_t^{(\varepsilon,\delta)}|^2 
           \nonumber &=\frac{2}{\delta}\int_{0}^{t}\langle \hat{Y}_s^{(\varepsilon,\delta)}-\tilde{Y}_s^{(\varepsilon,\delta)}, f_{2}(\hat{X}_{s}^{(\varepsilon, \delta)}, \hat{Y}_{s}^{(\varepsilon, \delta)})-f_{2}(\hat{X}_{s(\Delta)}^{(\varepsilon, \delta)}, \tilde{Y}_{s}^{(\varepsilon, \delta)})\rangle \d s\\
           \nonumber & \quad +\frac{2}{\sqrt{\varepsilon\delta}}\int_{0}^{t}\langle \hat{Y}_s^{(\varepsilon,\delta)}-\tilde{Y}_s^{(\varepsilon,\delta)},\sigma_{2}(\hat{X}_{s}^{(\varepsilon, \delta)}, \hat{Y}_{s}^{(\varepsilon, \delta)}) \dot{h}_2^{(\varepsilon, \delta)}(s)\rangle \d s\\
           &\quad +\frac{1}{\delta}\int_{0}^{t}\|\sigma_{2}(\hat{X}_{s}^{(\varepsilon, \delta)}, \hat{Y}_{s}^{(\varepsilon, \delta)})-\sigma_{2}(\hat{X}_{s(\Delta)}^{(\varepsilon, \delta)}, \tilde{Y}_{s}^{(\varepsilon, \delta)})\|^2\d s\\
          \nonumber &\quad +\frac{2}{\sqrt{\delta}}\int_{0}^{t}\langle  \hat{Y}_s^{(\varepsilon,\delta)}-\tilde{Y}_s^{(\varepsilon,\delta)},\big(\sigma_{2}(\hat{X}_{s}^{(\varepsilon, \delta)}, \hat{Y}_{s}^{(\varepsilon, \delta)})-\sigma_{2}(\hat{X}_{s(\Delta)}^{(\varepsilon, \delta)}, \tilde{Y}_{s}^{(\varepsilon, \delta)})\big)\d {W}_{s}^{2}\rangle\\
           \nonumber &=: \sum_{i=1}^{4}S_i.
        \end{align}
    By Assumption (\hyperlink{(H2)}{H2}), it follows that
    \begin{equation}\label{EQ:0729_02}
        S_1+S_3\leq \frac{1}{\delta} \int_{0}^{t}\left(-\beta_1|\hat{Y}_s^{(\varepsilon,\delta)}-\tilde{Y}_s^{(\varepsilon,\delta)}|^2+C|\hat{X}_{s}^{(\varepsilon, \delta)}-\hat{X}_{s(\Delta)}^{(\varepsilon, \delta)}|^2\right)\d s.
        \end{equation}
       By Assumption (\hyperlink{(H3)}{H3}) and Young’s inequality, 
    \begin{equation}\label{EQ:0729_03}
     \begin{split}
              S_2 &\leq \frac{2}{\sqrt{\varepsilon\delta}}\int_{0}^{t}\|\sigma_{2}(\hat{X}_{s}^{(\varepsilon, \delta)}, \hat{Y}_{s}^{(\varepsilon, \delta)}) \|\cdot|\dot{h}_2^{(\varepsilon, \delta)}(s)|\cdot|\hat{Y}_s^{(\varepsilon,\delta)}-\tilde{Y}_s^{(\varepsilon,\delta)}|\d s\\
        &\leq \frac{\gamma}{\delta}\int_{0}^{t}|\hat{Y}_s^{(\varepsilon,\delta)}-\tilde{Y}_s^{(\varepsilon,\delta)}|^2\d s+\frac{C}{\varepsilon}\int_{0}^{t}(1+|\hat{X}_s^{(\varepsilon,\delta)}|^2)\cdot|\dot{h}_2^{(\varepsilon, \delta)}(s)|^2\d s.
     \end{split}
    \end{equation}
    where we choose $\gamma\in (0,\beta_1)$. Plugging \eqref{EQ:0729_02} and \eqref{EQ:0729_03} into \eqref{EQ:0729_01} leads to
    \begin{equation*}
        \begin{split}
         |\hat{Y}_t^{(\varepsilon,\delta)}-\tilde{Y}_t^{(\varepsilon,\delta)}|^2
         &\leq -\frac{\beta_1-\gamma}{\delta}\int_{0}^{t} |\hat{Y}_s^{(\varepsilon,\delta)}-\tilde{Y}_s^{(\varepsilon,\delta)}|^2  \d s+\frac{C}{\delta}\int_{0}^{t}|\hat{X}_{s}^{(\varepsilon, \delta)}-\hat{X}_{s(\Delta)}^{(\varepsilon, \delta)}|^2\d s\\
         &\quad +\frac{2}{\sqrt{\delta}}\int_{0}^{t}\langle  \hat{Y}_s^{(\varepsilon,\delta)}-\tilde{Y}_s^{(\varepsilon,\delta)},\big(\sigma_{2}(\hat{X}_{s}^{(\varepsilon, \delta)}, \hat{Y}_{s}^{(\varepsilon, \delta)})-\sigma_{2}(\hat{X}_{s(\Delta)}^{(\varepsilon, \delta)}, \tilde{Y}_{s}^{(\varepsilon, \delta)})\big)\d {W}_{s}^{2}\rangle\\
         &\quad +\frac{C}{\varepsilon}\int_{0}^{t}(1+|\hat{X}_s^{(\varepsilon,\delta)}|^2)\cdot|\dot{h}_2^{(\varepsilon, \delta)}(s)|^2\d s.
        \end{split}
    \end{equation*}
    Taking expectation yields that
    \begin{align*}
        \frac{\beta_1-\gamma}{\delta}\E\int_{0}^{t}|\hat{Y}_s^{(\varepsilon,\delta)}-\tilde{Y}_s^{(\varepsilon,\delta)}|^2\d s
        &\leq \frac{C}{\delta}\E\int_{0}^{t}|\hat{X}_{s}^{(\varepsilon, \delta)}-\hat{X}_{s(\Delta)}^{(\varepsilon, \delta)}|^2\d s\\
        &\quad+\frac{C}{\varepsilon}\E \int_{0}^{t}(1+|\hat{X}_s^{(\varepsilon,\delta)}|^2)\cdot|\dot{h}_2^{(\varepsilon, \delta)}(s)|^2\d s
    \end{align*}
    Multiplying $\frac{\delta}{\beta_1-\gamma}$ on both sides of the above inequality, by Lemmas \ref{THM19}, \ref{THM17} and $h^{(\varepsilon,\delta)}\in \mathcal{A}_N,$ we have
    \begin{align*}
     \E\int_{0}^{T}|\hat{Y}_t^{(\varepsilon,\delta)}-\tilde{Y}_t^{(\varepsilon,\delta)}|^2\d t &\leq  C\E \int_{0}^{T}|\hat{X}_{t}^{(\varepsilon, \delta)}-\hat{X}_{t(\Delta)}^{(\varepsilon, \delta)}|^2\d t\\
       &\quad +\frac{C\delta}{\varepsilon}\E\int_{0}^{T}(1+|\hat{X}_t^{(\varepsilon,\delta)}|^2)\cdot|\dot{h}_2^{(\varepsilon, \delta)}(t)|^2\d t\\
       &\leq C\Delta+\frac{C\delta}{\varepsilon}\E\left[\sup_{t\leq T}(1+|\hat{X}_t^{(\varepsilon,\delta)}|^2)\cdot\int_{0}^{T}|\dot{h}_2^{(\varepsilon,\delta)}(t)|^2\d t\right]\\
       &\leq C(\Delta+\delta/\varepsilon).
    \end{align*}
    The proof is complete.
\end{proof}
\end{lemma}
The next lemma reveals the relationship between the controlled slow component and the auxiliary slow component.
 \begin{lemma}\label{LEM:0730:01}
Suppose that (\hyperlink{(H1)}{H1})-(\hyperlink{(H3)}{H3}) hold.  Let \( N\geq 1 \), 
\begin{equation}\label{EQ60}
        \begin{aligned}
            \E[|\hat{X}_t^{(\varepsilon,\delta)}-\tilde{X}_t^{(\varepsilon,\delta)}|^2]&\leq C\bigg[\Big(\varepsilon+\rho_{0,\eta_1}(\Delta )+\rho_{0,\eta_1}(\Delta+\delta/\varepsilon)\Big)
            \\
            &\qquad
            +\Big(\varepsilon+\rho_{0,\eta_1}(\Delta )+\rho_{0,\eta_1}(\Delta+\delta/\varepsilon)\Big)^{\exp \{-CT\}}\bigg].
        \end{aligned}
	\end{equation}
 \end{lemma}
 \begin{proof}
     It follows from \eqref{EQ12} and \eqref{EQ35} that 
	\begin{equation}
		\begin{split}
			\d (\hat{X}_t^{(\varepsilon,\delta)}-\tilde{X}_t^{(\varepsilon,\delta)}) & = (f_1(\hat{X}_t^{(\varepsilon,\delta)},\hat{Y}_t^{(\varepsilon,\delta)}) -f_1(\hat{X}_{t(\Delta)}^{(\varepsilon,\delta)},\tilde{Y}_t^{(\varepsilon,\delta)}))\d t \\
   &\quad + (\sigma_1(\hat{X}_t^{(\varepsilon,\delta)})-\sigma_1(\tilde{X}_t^{(\varepsilon,\delta)}))\dot{h}_1^{(\varepsilon,\delta)}(t)\d t + \sqrt{\varepsilon}\sigma_{1}(\hat{X}_t^{(\varepsilon,\delta)})\d W^1_t.
		\end{split}
	\end{equation}
    Set $J_1:= \E[ |\hat{X}_{t}^{(\varepsilon, \delta)}-\tilde{X}_{t}^{(\varepsilon,\delta)}|^2 ]$.  By the $C_r$-inequality and \hold's inequality, we have 
    \begin{align*}
    J_{1}& \leqslant C \E\left[\left|\int_{0}^{t} f_{1}(\hat{X}_{s}^{(\varepsilon, \delta)}, \hat{Y}_{s}^{(\varepsilon, \delta)})-f_{1}(\hat{X}_{s(\Delta)}^{(\varepsilon, s)}, \hat{Y}_{s}^{(\varepsilon, \delta)}) \d s\right|^{2}\right] +C \varepsilon \E\left[\left|\int_{0}^{t} \sigma_{1}(\hat{X}_{s}^{(\varepsilon, \delta)}) \d W_{s}^1\right|^{2}\right] \\
    &\quad+C \E\left[\left|\int_{0}^{t}(\sigma_{1}(\hat{X}_{s}^{(\varepsilon, \delta)})-\sigma_{1}(\tilde{X}_{s}^{(\varepsilon, \delta)})) \dot{h}_1^{(\varepsilon, \delta)}(s) \d s\right|^{2}\right] \\
    &\leqslant C \E\left[\int_{0}^{t}|f_{1}(\hat{X}_{s}^{(\varepsilon, \delta)}, \hat{Y}_{s}^{(\varepsilon, \delta)})-f_{1}(\hat{X}_{s(\Delta)}^{(\varepsilon)}, \hat{Y}_{s}^{(\varepsilon, \delta)})|^{2} \d s\right] +C \varepsilon \E\left[\int_{0}^{t}\|\sigma_{1}(\hat{X}_s^{(\varepsilon, \delta)})\|^{2} \d s\right]\\
    &\quad +C \E\left[\int_{0}^{t}\|\sigma_{1}(\hat{X}_{s}^{(\varepsilon, \delta)})-\sigma_{1}(\tilde{X}_{s}^{(\varepsilon, s)})\|^{2} \d s\cdot\int_{0}^{t}|\dot{h}_1^{(\varepsilon, \delta)}(s)|^{2} \d s\right] \\
    &\leqslant C\varepsilon+C \E\left[\int_{0}^{t}|f_{1}(\hat{X}_{s}^{(\varepsilon, \delta)}, \hat{Y}_{s}^{(\varepsilon, \delta)})-f_{1}(\hat{X}_{s(\Delta)}^{(\varepsilon)}, \hat{Y}_{s}^{(\varepsilon, \delta)})|^{2} \d s\right] \\
    &\quad +C \E\left[\int_{0}^{t}\|\sigma_{1}(\hat{X}_{s}^{(\varepsilon, \delta)})-\sigma_{1}(\tilde{X}_{s}^{(\varepsilon, \delta)})\|^{2} \d s\right] \\
    &=:C\varepsilon+J_{1,1}+J_{1,2}
    \end{align*}
    where we have used the linear growth of $ \sigma_1$ and Lemma \ref{THM19} in the last inequality.

	
   For $J_{1,1}$, by (\hyperlink{(H1)}{H1}),
   \begin{align*}
		J_{1,1}& \leq C \E \left[\int_{0}^{t} |f_1(\hat{X}_s^{(\varepsilon,\delta)},\hat{Y}_s^{(\varepsilon,\delta)}) 
		-f_1(\hat{X}_{s(\Delta)}^{(\varepsilon,\delta)},\hat{Y}_s^{(\varepsilon,\delta)})|^2 \d s\right]\\
		&\quad+ C\E \left[\int_{0}^{t} |f_1(\hat{X}_{s(\Delta)}^{(\varepsilon,\delta)},\hat{Y}_s^{(\varepsilon,\delta)}) -f_1(\hat{X}_{s(\Delta)}^{(\varepsilon,\delta)},\tilde{Y}_s^{(\varepsilon,\delta)})|^2\d s \right]\\
		& \leq C\E \left[\int_{0}^{t} |\hat{X}_s^{(\varepsilon,\delta)}-\hat{X}_{s(\Delta)}^{(\varepsilon,\delta)}|^2\cdot g_1^2(|\hat{X}_s^{(\varepsilon,\delta)}-\hat{X}_{s(\Delta)}^{(\varepsilon,\delta)}|)\d s\right] \\
		&\quad+C\E \left[\int_{0}^{t} 
		 |\hat{Y}_s^{(\varepsilon,\delta)}-\tilde{Y}_{s}^{(\varepsilon,\delta)}|^2\cdot g_1^2(|\hat{Y}_s^{(\varepsilon,\delta)}-\tilde{Y}_{s}^{(\varepsilon,\delta)}|) \d s\right]\\
		& \leq C\E\left[\int_{0}^{t} \rho_{\eta_1}^2 (|\hat{X}_s^{(\varepsilon,\delta)}-\hat{X}_{s(\Delta)}^{(\varepsilon,\delta)}|) \d s\right] + C\E\left[\int_{0}^{t} \rho_{\eta_1}^2 (|\hat{Y}_s^{(\varepsilon,\delta)}-\tilde{Y}_{s}^{(\varepsilon,\delta)}|^2 \d s \right]\\
		& \leq C \int_{0}^{t} \rho_{0,\eta_1}(\E |\hat{X}_s^{(\varepsilon,\delta)}-\hat{X}_{s(\Delta)}^{(\varepsilon,\delta)}|^2) \d s+C \int_{0}^{t} \rho_{0,\eta_1}(\E |\hat{Y}_s^{(\varepsilon,\delta)}-\tilde{Y}_{s}^{(\varepsilon,\delta)}|^2) \d s\\
		& \leq C\rho_{0,\eta_1}(\Delta )+C\rho_{0,\eta_1}(\Delta+\delta/\varepsilon),
	\end{align*}
	where the last inequality has used Jensen's inequality, Lemmas \ref{THM17} and \ref{LEM:0729:01}.
 

    By Remark \ref{RMK:01}, there exist two positive constants $C$ and $\eta_{3}$ such that for any $x \in \mathbb{R}_{+}$,  $x^2g_2(x)\leq C\rho_{\eta_{3}}(x^2)$. Hence, for $J_{1,2}$, it deduces from   (\hyperlink{(H1)}{H1}) that
	\begin{align*}
			J_{1,2}  &\leq C\E\left[\int_{0}^{t}|\hat{X}_s^{(\varepsilon,\delta)}-\tilde{X}^{(\varepsilon,\delta)}_s|^2\cdot g_2(|\hat{X}_s^{(\varepsilon,\delta)}-\tilde{X}^{(\varepsilon,\delta)}_s|) \d s\right]\\
			& \leq C\E\left[\int_{0}^{t}\rho_{\eta_3}(|\hat{X}_s^{(\varepsilon,\delta)}-\tilde{X}^{(\varepsilon,\delta)}_s|^2) \d s\right] \\
        &\leq C\int_{0}^{t}\rho_{\eta_3}(\E|\hat{X}_s^{(\varepsilon,\delta)}-\tilde{X}^{(\varepsilon,\delta)}_s|^2) \d s\\
        &\leq C\int_{0}^{t}\rho_{\eta_3}(\E|\hat{X}_s^{(\varepsilon,\delta)}-\tilde{X}^{(\varepsilon,\delta)}_s|^2)\d s.
		\end{align*}

Combining the above calculations implies 
\begin{equation*}
    J_1\leq C\varepsilon+C\rho_{0,\eta_1}(\Delta )+C\rho_{0,\eta_1}(\Delta+\delta/\varepsilon)+C\int_{0}^{t}\rho_{\eta_3}(\E|\hat{X}_s^{(\varepsilon,\delta)}-\tilde{X}^{(\varepsilon,\delta)}_s|^2)\d s,
\end{equation*}
which, together with Lemma \ref{THM02}-(1), leads to
the desired result.
 \end{proof}
 \begin{lemma}\label{LEM:0730:02}
Suppose that (\hyperlink{(H1)}{H1})-(\hyperlink{(H3)}{H3}) hold.  Let \( N\geq 1 \), for any fixed $x\in \R^d,$
\begin{equation*}
\begin{split}
     \E[|\tilde{X}_{t}^{(\varepsilon, \delta)}(x)-{X}_{t}^{h^{(\varepsilon,\delta)}}(x)|^2]&\leq  G^{-1}\Big(G\big(C\left(\rho_{\eta_1}({\delta}{\Delta}^{-1})+\rho_{0,\eta_1}({\delta}{\Delta}^{-1})\right.\\
&\qquad\qquad\left.+\Delta^2\rho_{0,\eta_1}( \Delta)+\rho_{0,\eta_1}(R(\varepsilon,\delta,\Delta)\right)\big)+CT\Big),
\end{split}
\end{equation*}     
  where $G(x)=\int_1^{x}\frac{1}{\rho_{0,\eta_1}(y)}\d y$ is strictly increasing and satisfies $G(x) \to -\infty$ as $x \downarrow 0$.
 \end{lemma}
\begin{proof}
Set $J_2:= \E[|\tilde{X}_{t}^{(\varepsilon, \delta)}(x)-{X}_{t}^{h^{(\varepsilon,\delta)}}(x)|^2].$ For simplicity of notation, the initial values $x$ of the processes will be omitted in this proof. By \eqref{EQ06} and  \eqref{EQ35},
	\begin{equation*}
		\begin{split}
			J_2 & = \E\left| \int_{0}^{t} f_1(\hat{X}_{s(\Delta)}^{(\varepsilon,\delta)},\tilde{Y}_s^{(\varepsilon,\delta)})+\sigma_1(\tilde{X}_s^{(\varepsilon,\delta)})\dot{h}_1^{(\varepsilon,\delta)}(s) -  \bar{f}_1(X_s^{h^{(\varepsilon,\delta)}}) - \sigma_1(X_s^{h^{(\varepsilon,\delta)}}) \dot{h}_1^{(\varepsilon,\delta)}(s)\d s \right|^2\\
			& \leq 3\E \left| \int_{0}^{t} f_1(\hat{X}_{s(\Delta)}^{(\varepsilon,\delta)},\tilde{Y}_s^{(\varepsilon,\delta)})-\bar{f}_1(\hat{X}_{s(\Delta)}^{(\varepsilon,\delta)}) \d s \right|^2 +
			3\E  \left|  \int_{0}^{t} \bar{f}_1(\hat{X}_{s(\Delta)}^{(\varepsilon,\delta)})-  \bar{f}_1(X_s^{h^{(\varepsilon,\delta)}})\d s\right|^2\\
			& \quad + 3\E\left|\int_{0}^{t} [\sigma_{1} (\tilde{X}_s^{(\varepsilon,\delta)}) -\sigma_{1}(X_s^{h^{(\varepsilon,\delta)}})] \dot{h}_1^{(\varepsilon,\delta)}(s)\d s \right|^2\\
   &=: \sum_{i=1}^{3} J_{2,i}.
		\end{split}
	\end{equation*} 
	We estimate these terms one by one. For \( J_{2,1} \), by $ C_r $-inequality, (\hyperlink{(H3)}{H3}), the linear growth of  $ \bar{f_1}$, and Lemma \ref{THM19}, we deduce 
	\begin{align*}
			J_{2,1} & \leq  C\E \bigg[ \bigg|\sum_{k=0}^{\lfloor\frac{t}{\Delta}\rfloor-1}\int_{k\Delta}^{(k+1)\Delta}f_1(\hat{X}_{k\Delta}^{(\varepsilon,\delta)},\tilde{Y}_s^{(\varepsilon,\delta)})-\bar{f}_1(\hat{X}_{k \Delta}^{(\varepsilon,\delta)}) \d s \bigg|^2\bigg] \\
			& \quad+  C\E \bigg[\bigg| \int_{\lfloor\frac{t}{\Delta}\rfloor\Delta}^{t} f_1(\hat{X}_{\lfloor\frac{t}{\Delta}\rfloor\Delta}^{(\varepsilon,\delta)},\tilde{Y}_s^{(\varepsilon,\delta)})-\bar{f}_1(\hat{X}_{\lfloor\frac{t}{\Delta}\rfloor\Delta}^{(\varepsilon,\delta)}) \d s\bigg|^2\bigg]\\
  & \leq  C \lfloor\frac{t}{\Delta}\rfloor \mathbb{E}\bigg[ \sum_{k=0}^{\left\lfloor\frac{t}{\Delta}\right\rfloor-1}\bigg|\int_{k \Delta}^{(k+1) \Delta}f_1(\hat{X}_{k\Delta}^{(\varepsilon,\delta)},\tilde{Y}_s^{(\varepsilon,\delta)})-\bar{f}_1(\hat{X}_{k \Delta}^{(\varepsilon,\delta)}) \d s\bigg|^2\bigg]  +C\Delta^2 \\
		&  	\leq \frac{C}{\Delta^2} \max _{0 \leq k \leq\left\lfloor\frac{T}{\Delta}\right\rfloor-1} \mathbb{E}\bigg[\bigg|\int_{k \Delta}^{(k+1) \Delta}f_1(\hat{X}_{k\Delta}^{(\varepsilon,\delta)},\tilde{Y}_s^{(\varepsilon,\delta)})-\bar{f}_1(\hat{X}_{k \Delta}^{(\varepsilon,\delta)})  \d s\bigg|^2\bigg]+C\Delta^2 \\
		& 	\leq  C \frac{\delta^2}{\Delta^2} \max _{0 \leq k \leq\left\lfloor\frac{T}{\Delta}\right\rfloor-1} \int_{0}^{\frac{\Delta}{\delta}} \int_{\zeta}^{\frac{\Delta}{\delta}} \mathcal{J}_{k}(s, \zeta) \d s \d \zeta+C \Delta^2,
		\end{align*}
	where \( 0 \leq \zeta \leq s \leq \frac{\Delta}{\delta} \) and 
	\begin{equation*}\label{EQ33}
		\mathcal{J}_{k}(s, \zeta):=\mathbb{E}[\langle f_{1}(\hat{X}_{k \Delta}^{(\varepsilon, \delta)}, \tilde{Y}_{s \delta+k \Delta}^{(\varepsilon, \delta)})-\bar{f}_{1}(\hat{X}_{k \Delta}^{(\varepsilon, \delta)}), f_{1}(\hat{X}_{k \Delta}^{(\varepsilon, \delta)}, \tilde{Y}_{\zeta \delta+k \Delta}^{(\varepsilon, \delta)})-\bar{f}_{1}(\hat{X}_{k \Delta}^{(\varepsilon, \delta)})\rangle].
	\end{equation*}
To proceed, let us consider the following equation, 
\begin{equation}\label{EQ:0711:01}
    \begin{cases}
        \d Y_t=\frac{1}{\delta} f_2(X,Y_t)\d t+\frac{1}{\sqrt{\delta}}\sigma_2(X,Y_t) \d W_t^2,\quad t\geq s,\\
        Y_s =Y,
    \end{cases}
\end{equation}
where $X,Y$ are $\mathscr{F}_s$-measurable $\R^d$-valued random variables. Denote by $\{\bar{Y}_{s,t}^X(Y)\}_{t\geq s}$ the unique solution of SDE \eqref{EQ:0711:01}. For each $k\in \N,$ and $t\in [k\Delta,(k+1)\Delta],$ we have 
\begin{equation*}\label{EQ:0711:02}
    \tilde{Y}_t^{(\varepsilon,\delta)} =\bar{Y}^{\hat{X}^{(\varepsilon,\delta)}_{k\Delta}}_{k\Delta,t}(\tilde{Y}^{(\varepsilon,\delta)}_{k\Delta}).
\end{equation*}
Since $\hat{X}^{(\varepsilon,\delta)}_{k\Delta}$ and $\tilde{Y}^{(\varepsilon,\delta)}_{k\Delta}$ are $\mathscr{F}_{k\Delta}$-measurable and $\{Y^x_{k\Delta,s\delta+k\Delta}(y)\}_{s\geq 0}$ is independent of $\mathscr{F}_{k\Delta}$ for each fixed $x,y\in \R^d,$  then
\begin{equation*}\label{EQ:0711:03}
    \begin{split}
        \mathcal{J}_k&(s,\zeta) \\
        &= \E[\E[\langle f_1(\hat{X}^{(\varepsilon,\delta)}_{k\Delta},\bar{Y}^{\hat{X}^{(\varepsilon,\delta)}_{k\Delta}}_{k\Delta,k\Delta+s\delta}(\tilde{Y}^{(\varepsilon,\delta)}_{k\Delta}))-\bar{f}_1(\hat{X}^{(\varepsilon,\delta)}_{k\Delta}), f_1(x,\bar{Y}^{\hat{X}^{(\varepsilon,\delta)}_{k\Delta}}_{k\Delta,k\Delta+\zeta\delta}(\tilde{Y}^{(\varepsilon,\delta)}_{k\Delta}))-\bar{f}_1(\hat{X}^{(\varepsilon,\delta)}_{k\Delta}\rangle ]]\\
        & = \E[\E[\langle f_1(\hat{X}^{(\varepsilon,\delta)}_{k\Delta},\bar{Y}^{\hat{X}^{(\varepsilon,\delta)}_{k\Delta}}_{k\Delta,k\Delta+s\delta}(\tilde{Y}^{(\varepsilon,\delta)}_{k\Delta}))-\bar{f}_1(\hat{X}^{(\varepsilon,\delta)}_{k\Delta}),\\
        &\qquad\qquad f_1(x,\bar{Y}^{\hat{X}^{(\varepsilon,\delta)}_{k\Delta}}_{k\Delta,k\Delta+\zeta\delta}(\tilde{Y}^{(\varepsilon,\delta)}_{k\Delta}))-\bar{f}_1(\hat{X}^{(\varepsilon,\delta)}_{k\Delta}\rangle \mid \mathscr{F}_{k\Delta}]]\\
        & = \E[\E[\langle f_1(x,\bar{Y}^{x}_{k\Delta,k\Delta+s\delta}(y))-\bar{f}_1(x),f_1(x,\bar{Y}^{x}_{k\Delta,k\Delta+\zeta\delta}(y))-\bar{f}_1(x)\rangle\big|_{(x,y)=(\hat{X}^{(\varepsilon,\delta)}_{k\Delta},\tilde{Y}^{(\varepsilon,\delta)}_{k\Delta})}]].
    \end{split}
\end{equation*}
 Since $\{\bar{Y}^x_{k\Delta,s\delta+k\Delta}(y)\}$ is the solution of SDE \eqref{EQ:0711:01}, it follows that 
 \begin{equation*}
     \begin{split}
   \bar{Y}^x_{k\Delta,s\delta+k\Delta}(y) & = y+\frac{1}{\delta}\int_{k\Delta}^{k\Delta+s\delta} f_2(x,\bar{Y}^x_{k\Delta,r}(y))\d r+ \frac{1}{\sqrt{\delta}}\int_{k\Delta}^{k\Delta+s\delta} \sigma_2(x,\bar{Y}^x_{k\Delta,r}(y)) \d W_r^2\\
   & =y+\frac{1}{\delta}\int_{0}^{s\delta} f_2(x,\bar{Y}^x_{k\Delta,r+k\Delta}(y))\d r +\frac{1}{\sqrt{\delta}}\int_{0}^{s\delta}\sigma_2(x,\bar{Y}^x_{k\Delta,r+k\Delta}(y))\d W^{2,k\Delta}_r\\
   & = y+ \int_{0}^{s}f_2(x,\bar{Y}^x_{k\Delta,r\delta+k\Delta}(y))\d r+\int_{0}^s\sigma_2(x,\bar{Y}^x_{k\Delta,r\delta+k\Delta}(y))\d \bar{W}^{2,k\Delta}_r,
     \end{split}
 \end{equation*}
 where $W^{2,k\Delta}_r:=W_{r+k\Delta}^2-W_{k\Delta}^2$ and $\bar{W}^{2,k\Delta}_r:=\frac{1}{\sqrt{\delta}}W^{2,k\Delta}_{r\delta}.$
The uniqueness of the solution to the frozen process \eqref{EQ05} yields that 
\begin{equation*}
    \{\bar{Y}^x_{k\Delta,r\delta+k\Delta}(y)\}_{0\leq s\leq \frac{\Delta}{\delta}} \stackrel{d}{=} \{Y_s^x(y)\}_{0\leq s\leq \frac{\Delta}{\delta}},
\end{equation*}
where \( \stackrel{d}{=}\) denotes coincidence in the distribution sense. Thus, we can rewrite $J_k(s,\zeta)$ as follows
\begin{equation}\label{EQ:0711:05}
    \begin{split}
      \mathcal{J}_k(s,\zeta)
      &= \E[\E[\langle f_1(x,Y^{x}_{s}(y))-\bar{f}_1(x),f_1(x,Y^{x}_{\zeta}(y))-\bar{f}_1(x)\rangle\big|_{(x,y)=(\hat{X}^{(\varepsilon,\delta)}_{k\Delta},\tilde{Y}^{(\varepsilon,\delta)}_{k\Delta})}]]\\
      &=: \E[\mathcal{J}_k(s,\zeta,x,y)\big|_{(x,y)=(\hat{X}^{(\varepsilon,\delta)}_{k\Delta},\tilde{Y}^{(\varepsilon,\delta)}_{k\Delta})}],
    \end{split}
\end{equation}
for any \( 0 \leq \zeta \leq s \leq \frac{\Delta}{\delta} \). Then due to the Markov property and the time-homogeneity of process $Y^{x}_{\cdot}(y),$  we obtain
\begin{equation*}
    \begin{split}
		\mathcal{J}_{k}(s, \zeta, x, y)
		& =\mathbb{E}[\mathbb{E}[\langle f_{1}(x, Y_{s}^{x}(y))-\bar{f}_{1}(x),(f_{1}(x, Y_{\zeta}^{x}(y))-\bar{f}_{1}(x))\rangle\mid\mathscr{F}_{\zeta}^{x}]] \\
		& =\mathbb{E}[\langle  \mathbb{E}[(f_{1}( x, Y_{s}^{x}(y))-\bar{f}_{1}(x)) \mid \mathscr{F}_{\zeta}^{x}],f_{1}\left( x, Y_{\zeta}^{x}(y)\right)-\bar{f}_{1}(x)\rangle]\\
  &=\mathbb{E}\left[ \left\langle \mathbb{E}(f_{1}( x, Y_{s-\zeta}^{x}(Y_{\zeta}^{x}(y)))-\bar{f}_{1}(x)),f_{1}(x, Y_{\zeta}^{x}(y))-\bar{f}_{1}(x)\right\rangle\right].
    \end{split}
\end{equation*}
 Furthermore, using  (\hyperlink{H1}{H1}), \eqref{EQ25} and Lemma \ref{THM05}, we have
	\begin{align}\label{EQ30}
		\nonumber \mathcal{J}_{k}&(s, \zeta, x, y)\\
  \nonumber & \leq  \mathbb{E}\left[\left|\mathbb{E}(f_{1}( x, Y_{s-\zeta}^{x}(Y_{\zeta}^{x}(y))-\bar{f}_{1}( x))\right|\cdot\left|f_{1}\left( x, Y_{\zeta}^{x}(y)\right)-\bar{f}_{1}(x)\right|\right]\\
  & \leq C\E\left[\rho_{\eta_1}\big(\e^{-\frac{\beta_{1}}{2}(s-\zeta)}(1+|x|+|Y_{\zeta}^{x}(y)|)\big)\big (1+\rho_{\eta_1}(|x|+|Y_{\zeta}^{x}(y)|)+\rho_{\eta_1}(|x|)\big)\right]\\
  \nonumber&\leq C\E\left[\rho_{\eta_1}\big(\e^{-\frac{\beta_{1}}{2}(s-\zeta)}(1+|x|+|Y_{\zeta}^{x}(y)|)\big)\right]+C\E\left[\rho_{\eta_1}^2\big(\e^{-\frac{\beta_{1}}{2}(s-\zeta)}(1+|x|+|Y_{\zeta}^{x}(y)|)\big)\right]\\
  \nonumber &\leq C\rho_{\eta_1}\big(\e^{-\frac{\beta_{1}}{2}(s-\zeta)}(1+|x|+\E[|Y_{\zeta}^{x}(y)|])\big)+C\rho_{0,\eta_1}\Big(\e^{-\beta_{1}(s-\zeta)}(1+|x|^2+\E[|Y_{\zeta}^{x}(y)|^2])\Big)\\
  \nonumber &\leq C\rho_{\eta_1}\big(\e^{-\frac{\beta_{1}}{2}(s-\zeta)}(1+|x|+|y|)\big)+C\rho_{0,\eta_1}\Big(\e^{-\beta_{1}(s-\zeta)}(1+|x|^2+|y|^2)\Big)
	\end{align}
	where $\rho_{0,\eta_1}$  is the corresponding transform function of $\rho_{\eta_1}$ in Lemma \ref{LEM:CONCAVE:01} and \eqref{EQ24} has been used in the last inequality. With the aid of \eqref{EQ:0711:05} and \eqref{EQ30}  yields that
	\begin{equation*}\label{EQ32}
	\begin{split}
	    	\mathcal{J}_{k}(s, \zeta)& \leq C \mathbb{E}[\rho_{\eta_1}(\e^{-\frac{\beta_{1}}{2}(s-\zeta)} )(1+|\hat{X}_{k \Delta}^{(\varepsilon,\delta)}|+|\tilde{Y}_{k \Delta}^{(\varepsilon,\delta)}|) ]\\
      &\quad+C\E\left[\rho_{0,\eta_1}\Big(\e^{-\beta_{1}(s-\zeta)}(1+|\hat{X}_{k \Delta}^{(\varepsilon,\delta)}|^2+|\tilde{Y}_{k \Delta}^{(\varepsilon,\delta)}|^2)\Big)\right]\\
      & \leq C\rho_{\eta_1}(\e^{-\frac{\beta_{1}}{2}(s-\zeta)} (1+\E[|\hat{X}_{k \Delta}^{(\varepsilon,\delta)}|]+\E[|\tilde{Y}_{k \Delta}^{(\varepsilon,\delta)}|]) )\\
      &\quad+C\rho_{0,\eta_1}\Big(\e^{-\beta_{1}(s-\zeta)}(1+\E[|\hat{X}_{k \Delta}^{(\varepsilon,\delta)}|^2]+\E[|\tilde{Y}_{k \Delta}^{(\varepsilon,\delta)}|^2])\Big)\\
      &\leq C\rho_{0,\eta_1}(\e^{-\beta_{1}(s-\zeta)})+C\rho_{0,\eta_1}(\e^{-\beta_{1}(s-\zeta)})
	\end{split}
	\end{equation*}
	Combining the above equation, due to Jensen's inequality  and the concavity of $ \rho_{\eta_1}(\cdot), $ 
		\begin{align*}
			J_{2,1} & \leq C \frac{\delta^2}{\Delta^2} \max _{0 \leq k \leq\left\lfloor\frac{T}{\Delta}\right\rfloor-1} \int_{0}^{\frac{\Delta}{\delta}} \int_{\zeta}^{\frac{\Delta}{\delta}} \mathcal{J}_{k}(s, \zeta) \d s \d \zeta+C\Delta^2 \\
			& \leq C \frac{\delta^2}{\Delta^2} \int_{0}^{\frac{\Delta}{\delta}} \int_{\zeta_1}^{\frac{\Delta}{\delta}}\Big(\rho_{\eta_1}(\e^{-\frac{\beta_{1}}{2}(s-\zeta)})+ \rho_{0,\eta_1}(\e^{-\beta_{1}(s-\zeta)})\Big)\d s \d \zeta+C \Delta^2 \\
			& \leq C\left(\rho_{\eta_1}\Big(\frac{2\delta^2}{\Delta^2}\int_{0}^{\frac{\Delta}{\delta}}\int_{\zeta}^{\frac{\Delta}{\delta}}\e^{-\frac{\beta_{1}}{2}(s-\zeta)}\d s \d \zeta \Big)+\rho_{0,\eta_1}\Big(\frac{2\delta^2}{\Delta^2}\int_{0}^{\frac{\Delta}{\delta}}\int_{\zeta}^{\frac{\Delta}{\delta}}\e^{-\beta_{1}(s-\zeta)}\d s \d \zeta \Big)\right)+C\Delta^2\\
			& \leq C\left(\rho_{\eta_1}\Big(\frac{2\delta^2}{\Delta^2}\Big(\frac{2}{\beta_{1}} \frac{\Delta}{\delta}-\frac{4}{\beta_{1}^{2}}+\e^{-\frac{\beta_{1}}{2} \frac{\Delta}{\delta}}\Big)\Big)+\rho_{0,\eta_1}\Big(\frac{2\delta^2}{\Delta^2}\Big(\frac{1}{\beta_{1}} \frac{\Delta}{\delta}-\frac{1}{\beta_{1}^{2}}+\e^{-\beta_{1}\frac{\Delta}{\delta}}\Big)\right)+C\Delta^2\\
			& \leq C(\rho_{\eta_1}({\delta}{\Delta}^{-1})+\rho_{0,\eta_1}({\delta}{\Delta}^{-1})+\Delta^2),
		\end{align*}
	where the final inequality comes from choosing \( \Delta=\Delta(\delta) \) so that \( {\delta}{\Delta}^{-1} \) is sufficiently small.
	
Subsequently, it proceeds to estimate $ J_{2,2} $. It follows from Lemmas \ref{THM17} and \ref{LEM:0730:01} that
	\begin{align}\label{EQ:0711:06}
			J_{2,2}
   \nonumber & \leq  C \E \bigg|  \int_{0}^{t} \bar{f}_1(\hat{X}_{s(\Delta)}^{(\varepsilon,\delta)})-  \bar{f}_1(\hat{X}_s^{(\varepsilon,\delta)})\d s\bigg|^2 + C \E \bigg|  \int_{0}^{t} \bar{f}_1(\hat{X}_{s}^{(\varepsilon,\delta)})-  \bar{f}_1(\tilde{X}_s^{(\varepsilon,\delta)})\d s\bigg|^2\\ 
	\nonumber		& \qquad + C \E \bigg|  \int_{0}^{t} \bar{f}_1(\tilde{X}_{s}^{(\varepsilon,\delta)})-  \bar{f}_1(X_s^{h^{(\varepsilon,\delta)}})\d s\bigg|^2\\
 \nonumber  &\leq C \int_{0}^{t} \mathbb{E}[|\bar{f}_{1}(\hat{X}_{s(\Delta)}^{(\varepsilon, \delta)})-\bar{f}_{1}(\hat{X}_{s}^{(\varepsilon, \delta)})|^2] \d s+C \int_{0}^{t} \mathbb{E}[|\bar{f}_{1}(\hat{X}_{s}^{(\varepsilon, \delta)})-\bar{f}_{1}(\tilde{X}_{s}^{(\varepsilon, \delta)})|^2] \d s \\
    \nonumber &\qquad +C   \int_{0}^{t} \E[|\bar{f}_1(\tilde{X}_{s}^{(\varepsilon,\delta)})-  \bar{f}_1(X_s^{h^{(\varepsilon,\delta)}})|^2]\d s\\
	\nonumber	& \leq C \int_{0}^{t} \mathbb{E}\left[\rho_{\eta_1}^2(|\hat{X}_{s(\Delta)}^{(\varepsilon, \delta)}-\hat{X}_{s}^{(\varepsilon, \delta)}|)+\rho_{\eta_1}^2(|\hat{X}_{s}^{(\varepsilon, \delta)}-\tilde{X}_{s}^{(\varepsilon, \delta)}|)+\rho_{\eta_1}^2(|\tilde{X}_{s}^{(\varepsilon,\delta)}-  X_s^{h^{(\varepsilon,\delta)}}|) \right] \d s \\
		& \leq C \int_{0}^{t} \mathbb{E}[\rho_{0,\eta_1}(|\hat{X}_{s}^{(\varepsilon, \delta)}-\hat{X}_{s(\Delta)}^{(\varepsilon, \delta)}|^2)+\rho_{0,\eta_1}(|\tilde{X}_{s}^{(\varepsilon, \delta)}-\hat{X}_{s}^{(\varepsilon, \delta)}|^2)\\
  \nonumber &\qquad\qquad+\rho_{0,\eta_1}(|\tilde{X}_{s}^{(\varepsilon,\delta)}-  X_s^{h^{(\varepsilon,\delta)}}|^2)]\d s\\
	\nonumber	& \leq C \int_{0}^{t} \left(\rho_{0,\eta_1}(\E[|\hat{X}_{s}^{(\varepsilon, \delta)}-\hat{X}_{s(\Delta)}^{(\varepsilon, \delta)}|^2])+\rho_{0,\eta_1}(\E[|\tilde{X}_{s}^{(\varepsilon, \delta)}-\hat{X}_{s}^{(\varepsilon, \delta)}|^2])\right. \\
   \nonumber &\qquad\qquad\left.+\rho_{0,\eta_1}(\E [|\tilde{X}_{s}^{(\varepsilon,\delta)}-  X_s^{h^{(\varepsilon,\delta)}}|^2])\right)\d s\\
		\nonumber & \leq C\left(\rho_{0,\eta_1}( \Delta)+\rho_{0,\eta_1}(R(\varepsilon,\delta,\Delta)\right)+C \int_{0}^{t} \rho_{0,\eta_1}(\E [|\tilde{X}_{s}^{(\varepsilon,\delta)}-  X_s^{h^{(\varepsilon,\delta)}}|^2])\d s,
 \end{align}
where 
\[
R(\varepsilon,\delta,\Delta)=\Big(\varepsilon+\rho_{0,\eta_1}(\Delta )+\rho_{0,\eta_1}(\Delta+\delta/\varepsilon)\Big)
            +\Big(\varepsilon+\rho_{0,\eta_1}(\Delta )+\rho_{0,\eta_1}(\Delta+\delta/\varepsilon)\Big)^{\exp \{-CT\}}
\]
 tends to 0 as $\varepsilon, \Delta\to 0$.
	 
  For $ J_{2,3}, $ by \hold's inequality,
	\begin{align}\label{EQ36}
			J_{2,3} \nonumber  & = \E\left|\int_{0}^{t} [\sigma_{1} (\tilde{X}_s^{(\varepsilon,\delta)}) -\sigma_{1}(X_s^{h^{(\varepsilon,\delta)}})] \dot{h}_1^{(\varepsilon,\delta)}(s)\d s \right|^2 \\ 
            \nonumber & \leq \E \left\{\left[\int_{0}^{t} \|\sigma_{1}(\tilde{X}_s^{(\varepsilon,\delta)})-\sigma_{1}(X_s^{h^{(\varepsilon,\delta)}})\|^2\d s \right] \cdot\left[\int_{0}^{t}|\dot{h}_1^{(\varepsilon,\delta)}(s)|^2\d s\right]\right\}\\ 
			&\leq {N^2}\int_{0}^{t} \E [\|\sigma_{1}(\tilde{X}_s^{(\varepsilon,\delta)})-\sigma_{1}(X_s^{h^{(\varepsilon,\delta)}})\|^2]\d s\\
			\nonumber  &\leq N^2\int_{0}^{t} \rho_{\eta_3}(\E [|\tilde{X}_{s}^{(\varepsilon,\delta)}-  X_s^{h^{(\varepsilon,\delta)}}|^2])\d s.
	\end{align}
	Combining the above calculations, we have 
	\begin{equation*}
		\begin{split}
			J_2   & \leq C(\rho_{\eta_1}({\delta}{\Delta}^{-1})+\rho_{0,\eta_1}({\delta}{\Delta}^{-1})+\Delta^2)+C\left(\rho_{0,\eta_1}( \Delta)+\rho_{0,\eta_1}(R(\varepsilon,\delta,\Delta)\right)\\
			& \quad +C \int_{0}^{t} \rho_{0,\eta_1}(\E [|\tilde{X}_{s}^{(\varepsilon,\delta)}-  X_s^{h^{(\varepsilon,\delta)}}|^2])\d s+ N^2\int_{0}^{t} \rho_{\eta_3}(\E [|\tilde{X}_{s}^{(\varepsilon,\delta)}-  X_s^{h^{(\varepsilon,\delta)}}|^2])\d s\\
			&\leq C(\rho_{\eta_1}({\delta}{\Delta}^{-1})+\rho_{0,\eta_1}({\delta}{\Delta}^{-1})+\Delta^2)+C\left(\rho_{0,\eta_1}( \Delta)+\rho_{0,\eta_1}(R(\varepsilon,\delta,\Delta)\right)\\
			& \quad +C\int_{0}^{t} \rho_{0,\eta_1}(\E [|\tilde{X}_{s}^{(\varepsilon,\delta)}-  X_s^{h^{(\varepsilon,\delta)}}|^2])\d s,\\
		\end{split}
	\end{equation*}
	where  we have used the fact $\rho_{\eta_3}(x)\leq C\rho_{0,,\eta_1}(x)$. By Lemma \ref{LEM:0801:02}, we can deduce that
	\begin{equation}\label{EQ62}
 \begin{split}
\E[|\tilde{X}_{t}^{(\varepsilon, \delta)}-{X}_{t}^{h^{(\varepsilon,\delta)}}|^2] & \leq  G^{-1}\Big(G\big(C\left(\rho_{\eta_1}({\delta}{\Delta}^{-1})+\rho_{0,\eta_1}({\delta}{\Delta}^{-1})\right.\\
&\qquad\qquad+\Delta^2+\rho_{0,\eta_1}( \Delta)+\rho_{0,\eta_1}(R(\varepsilon,\delta,\Delta))\big)\big)+CT\Big).
  \end{split}
	\end{equation}  
 The proof is complete.
\end{proof}
 Based on the above two lemmas, we have the following core estimate of the controlled slow component from two different initial values.
 \begin{proposition}\label{THM18}
		Suppose that (\hyperlink{(H1)}{H1})-(\hyperlink{(H3)}{H3}) hold.  Let \( N\geq 1 \) and \( f \) be a smooth non-decreasing function from  \( \mathbb{R}^{+} \) to  \( \mathbb{R}^{+} \) satisfying
			\[
			f^{\prime}(x) \leq 1 ; \quad f(x)=x,\; 0 \leq x<\frac{1}{4} ; \quad f(x)=1,\; x \geq 2 .
		\] 
		 Then there exist positive numbers \( p \) and \( C\)  independent of \( ~\varepsilon \) such that
		\[
		\mathbb{E}\left[\left(f(|\hat{X}_{t}^{(\varepsilon, \delta)}(x_1)-\hat{X}_{t}^{(\varepsilon,\delta)}(x_2)|)\right)^{p}\right] \leq C|x_1-x_2|^{2(d+1)},
		\]
		for all \( \varepsilon  \) sufficiently small and  $t\in [0,T]$.
	\end{proposition}
	\begin{proof}
		Notice that $ f(\cdot) $ satisfies that
	\[ 
	f(x+y)\leq f(x)+f(y),
	\] 
  for any $x,y\in \mathbb{R^+}.$  Let $p\geq 2$ be a positive number to be determined later. Then by $C_r$-inequality, we have that 
  \begin{equation}\label{EQ:0801:01}
	\begin{aligned}
	 	\E &\left[\left(f(|\hat{X}_{t}^{(\varepsilon, \delta)}(x_1)-\hat{X}_{t}^{(\varepsilon,\delta)}(x_2)|) \right)^p \right]\\
  & \leq 5^{p-1}\E\left[ \left(f(|\hat{X}_{t}^{(\varepsilon, \delta)}(x_1)-\tilde{X}_{t}^{(\varepsilon,\delta)}(x_1)|)\right)^p +\left( f(|\tilde{X}_{t}^{(\varepsilon, \delta)}(x_1)-{X}_{t}^{h^{(\varepsilon,\delta)}}(x_1)|) \right)^p\right.\\
		&\quad + \left(f(|{X}_{t}^{h^{(\varepsilon, \delta)}}(x_1)-{X}_{t}^{h^{(\varepsilon,\delta)}}(x_2)|)\right)^p+ \left(f(|\tilde{X}_{t}^{(\varepsilon, \delta)}(x_2)-{X}_{t}^{h^{(\varepsilon,\delta)}}(x_2)|) \right)^p\\
		& \quad + \left. \left(f(|\hat{X}_{t}^{(\varepsilon, \delta)}(x_2)-\tilde{X}_{t}^{(\varepsilon,\delta)}(x_2)|)\right)^p\right] \\
		& \leq 5^{p-1}\left(\E[ |\hat{X}_{t}^{(\varepsilon, \delta)}(x_1)-\tilde{X}_{t}^{(\varepsilon,\delta)}(x_1)|^2 +\E[|\tilde{X}_{t}^{(\varepsilon, \delta)}(x_1)-{X}_{t}^{h^{(\varepsilon,\delta)}}(x_1)|^2]\right.\\
		&\quad + \E\left[\left(f(|{X}_{t}^{h^{(\varepsilon, \delta)}}(x_1)-{X}_{t}^{h^{(\varepsilon,\delta)}}(x_2)|)\right)^p\right]+ \E[|\tilde{X}_{t}^{(\varepsilon, \delta)}(x_2)-{X}_{t}^{h^{(\varepsilon,\delta)}}(x_2)|^2]\\
		& \quad\left. + \E[ |\hat{X}_{t}^{(\varepsilon, \delta)}(x_2)-\tilde{X}_{t}^{(\varepsilon,\delta)}(x_2)|^2]\right),\\
	\end{aligned}
 \end{equation}
	where we have used $ (f(x))^p\leq x^2$, for any $x\geq 0,$  $p\geq 2$  in the second inequality and ${X}_{t}^{h^{(\varepsilon,\delta)}}$ is defined by ODE \eqref{EQ06}  with $h=h^{(\varepsilon,\delta)}$.

 First, we need to estimate the first term (or the last term)  and the second term (or the fourth term), which have been given in Lemmas \ref{LEM:0730:01} and \ref{LEM:0730:02}, respectively. Therefore, we only need to estimate the third term. 
 Set
	\[	p(t):=d+1+t+\int_{0}^{t}|\dot{h}^{(\varepsilon,\delta)}(s)| \d s.
	\]
	Since \( h^{(\varepsilon,\delta)} \in \mathcal{A}_{N} \) , we have that
	\[
	d+1 \leq p(t) \leq d+1+T+\sqrt{T}N=: A_{T, N}, \quad \forall t \in[0, T] .
	\]
	By \eqref{EQ27}, we can choose \( \delta_{1}=\delta_{1}(T, N) \) small enough such that
	\begin{equation}\label{EQ28}
		|\bar{f}_1(x)-\bar{f}_1(y)| \leq\left\{\begin{array}{ll}
			\frac{1}{2 A_{T, N}}|x-y| \log |x-y|^{-1}, & 0<|x-y| \leq \delta_{1}, \\
			C_{T, N}|x-y|, & |x-y|>\delta_{1},
		\end{array}\right.
	\end{equation}
	and
	\begin{equation}\label{EQ50}
		CA_{T, N} \sqrt{-\log x} \leq-\log x, \quad 0<x<\delta_{1} .
	\end{equation}
	Now put
	\[
	Z_{t}:=X_{t}^{h^{(\varepsilon,\delta)}}(x_1)-X_{t}^{h^{(\varepsilon,\delta)}}(x_2).
	\]
	By Newton-Leibniz's formula, we have 
	\[ 
f\left(\left|Z_{t}\right|\right)^{p(t)}=f\left(\left|Z_{0}\right|\right)^{p(0)}+\sum_{j=1}^{3} \int_{0}^{t} \xi_{j}(s) \d s,
	\]
	where
	\[ 
	\begin{aligned}
	    \xi_{1}(s)&:=(1+|\dot{h}_1^{(\varepsilon,\delta)}(s)|) f(|Z_{s}|)^{p(s)} \log f\left(\left|Z_{s}\right|\right), \\
		\xi_{2}(s)&:=\sum_{i=1}^{d} p(s) f(|Z_{s}|)^{p(s)-1} f^{\prime}(|Z_{s}|) \frac{Z_{s}^{i}}{\left|Z_{s}\right|}\big(\bar{f}_1^{i}(X_{s}^{h^{(\varepsilon,\delta)}}(x_1))-\bar{f}_1^{i}(X_{s}^{h^{(\varepsilon,\delta)}}(x_2))\big),\\
		\xi_{3}(s)&:=\sum_{i=1}^{d} p(s) f(|Z_{s}|)^{p(s)-1} f^{\prime}(|Z_{s}|) \frac{Z_{s}^{i}}{\left|Z_{s}\right|}\langle\sigma_1^{i}(X_{s}^{h^{(\varepsilon,\delta)}}(x_1))-\sigma_1^{i}(X_{s}^{h^{(\varepsilon,\delta)}}(x_2)),\dot{h}_1^{(\varepsilon,\delta)}(s)\rangle.
	\end{aligned}
	\]
	If \( \left|Z_{s}\right|<\delta_{1} \), then by \eqref{EQ28}  we have
	\[
	\xi_{2}(s) \leq-\frac{1}{2} f\left(|Z_{s}|\right)^{p(s)} \log |Z_{s}|.
	\]
	If \( \left|Z_{s}\right|<\delta_{1} \wedge \eta_2 \), then by \eqref{EQ50} we have
	\[
	\begin{aligned}
		\xi_{3}(s) & \leq  CA_{T, N} f\left(\left|Z_{s}\right|\right)^{p(s)} \sqrt{-\log \left|Z_{s}\right|} \cdot|\dot{h}_1^{(\varepsilon,\delta)}(s)| \\
		& \leq-f\left(Z_{s}\right)^{p(s)} \log \left|Z_{s}\right| \cdot|\dot{h}_1^{(\varepsilon,\delta)}(s)|,
	\end{aligned}
	\]
	which implies \[
	\xi_{1}(s)+\xi_{2}(s) +\xi_{3}(s) \leq 0 .
	\]	where we have used the fact $ f(x)\leq x$, $\forall x \geq 0$.
 
	If \( \left|Z_{s}\right| \geq \delta_{1} \), then it is easy to see that there exists a constant \( C \) such that
	\[
	\xi_{1}(s) + \xi_{2}(s)+\xi_{3}(s) \leq C(1+|\dot{h}_1^{(\varepsilon,\delta)}(s)|) f\left(|Z_{s}|\right)^{p(s)} .
	\]
	Hence, we always have
	\[
	f\left(\left|Z_{t}\right|\right)^{p(t)} \leq f\left(\left|Z_{0}\right|\right)^{d+1}+C \int_{0}^{t}(1+|\dot{h}_1^{(\varepsilon,\delta)}(s)|) f\left(\left|Z_{s}\right|\right)^{p(s)} \d s.
	\]
	Now by Hölder's inequality, we have
	\begin{align*}
		\mathbb{E}\left(f\left(\left|Z_{t}\right|\right)^{2 p(t)}\right)  & \leq 2 f\left(\left|Z_{0}\right|\right)^{2(d+1)}+C\mathbb{E}\left(\int_{0}^{t}(1+|\dot{h}_1^{(\varepsilon,\delta)}(s)|)^{2} \d s \cdot \int_{0}^{t} f\left(\left|Z_{s}\right|\right)^{2 p(s)} \d s\right) \\
		& 	\leq 2 f\left(\left|Z_{0}\right|\right)^{2(d+1)}+ C \int_{0}^{t} \mathbb{E}\left(f\left(\left|Z_{s}\right|\right)^{2 p(s)}\right) \d s.
	\end{align*}
	Since \( f(x) \leq 1 \), we obtain by  Gronwall's inequality that
	\begin{equation*}
		\begin{aligned}		\mathbb{E}\left(f\left(\left|Z_{t}\right|\right)^{2(d+1+T+\sqrt{T N})}\right) & \leq \mathbb{E}\left(f\left(\left|Z_{t}\right|\right)^{2 p(t)}\right) \leq C|x_1-x_2|^{2(d+1) }.
		\end{aligned}
	\end{equation*}
	 Hence, letting $p=2(d+1+T+\sqrt{T N}) $ implies the desired estimate
	\begin{equation}\label{EQ61}
		\mathbb{E}(f\left(\left|Z_{t}\right|\right)^p)\leq C|x_1-x_2|^{2(d+1)}.
	\end{equation}
	
  To sum up, by Lemmas \ref{LEM:0730:01}, \ref{LEM:0730:02}, \eqref{EQ:0801:01} and \eqref{EQ61}, we see that there exist two positive numbers \( p \) and \( C\) independent of \( \varepsilon \) such that
\[
\mathbb{E}\left[\big(f(|\hat{X}_{t}^{(\varepsilon, \delta)}(x_1)-\hat{X}_{t}^{(\varepsilon,\delta)}(x_2)|)\big)^{p}\right] \leq C|x_1-x_2|^{2(d+1)},
	\]
for all \( \varepsilon  \) sufficiently small and $t\in [0,T]$. The proof is complete.
\end{proof}

	\section{Proof of Main Theorem}\label{SEC04}
	\subsection{Rate function} 	
	Define a mapping $ S: \mathbb{H}\mapsto \C_T $ according to ODE \eqref{EQ06} by 
	\begin{equation}\label{EQ:0629:01}
	    	S(h)(t,x) :=X_t^h(x),
	\end{equation}
 which is  well-defined benefit from Theorem \ref{THM07}.
	It means that for any element  $ h $ in the Cameron-Martin Space $ \mathbb{H} $,  there is a corresponding skeleton function $S(h)$.  By the definition of the Laplace principle, we need to show that $I(f)$ defined by \eqref{EQ10} is indeed a rate function. To proceed, we first give  some warm lemmas.
	\begin{lemma}\label{THM09}
	For any \( N>0 \), the set \( \left\{S(h) ;\|h\|_{\mathbb{H}} \leq N\right\} \) is relatively compact in \( \mathbb{C}_{T} \).
	\end{lemma}
	\begin{proof}
 
	By the Ascoli-Arzelà theorem, we only need to prove the following two facts for any \( R \in \mathbb{N} \) :\\
		(i) \( \sup _{\|h\|_{\mathbb{H}} \leq N}\|S(h)\|_{\mathbb{C}_{T}^{R}}<+\infty \).\\
		(ii) \( \left\{(t, x) \mapsto S(h)(t, x) ;(t, x) \in[0, T] \times D_{R},\|h\|_{\mathbb{H}} \leq N\right\} \) is equi-continuous.
		
		Set $ \varphi_{h}(t):=|X_{t}^{h}(x)-X_{t}^{h}(y)|. $
		By \eqref{EQ06} and \eqref{EQ67},  we have
  \begin{equation}\label{eq:08021}
	\begin{aligned}
		\varphi_{h}(t) & \leq  |x-y|+\left|\int_{0}^{t}(\sigma_1(X_{s}^{h}(x))-\sigma_1(X_{s}^{h}(y)))\dot{h}_1({s}) \d s\right|  +\left|\int_{0}^{t}\bar{f}_1(X_{s}^{h}(x))-\bar{f}_1(X_{s}^{h}(y))\d s\right| \\
		& \leq  |x-y|+C\int_{0}^{t} \sqrt{ |X_s^h(x)-X_s^h(y)|^2\cdot g_2(|X_s^h(x)-X_s^h(y)|)} \cdot |\dot{h}_1(s)|\d s\\
		& \quad +C\int_{0}^{t} \rho_{\eta_1}(|X_s^h(x)-X_s^h(y)|) \d s \\
  & \leq  |x-y|+C\int_{0}^{t} \sqrt{ |X_s^h(x)-X_s^h(y)|\cdot \rho_{\eta_2}(|X_s^h(x)-X_s^h(y)|)} \cdot |\dot{h}_1(s)|\d s\\
		& \quad +C\int_{0}^{t} \rho_{\eta_1}(|X_s^h(x)-X_s^h(y)|) \d s \\
		&  \leq |x-y|+C\int_{0}^{t} \rho_{\eta^{\prime}}(\varphi_{h}(s)) \cdot (1+|\dot{h}_1({s})|) \d s,
	\end{aligned}
  \end{equation}
		where \( 0<\eta^{\prime}<\eta_1\wedge\eta_2 \) is small enough. By Theorem \ref{THM02}-(2), we obtain that 
	\begin{equation}\label{EQ11}
			\begin{aligned}
			\sup _{0 \leq t \leq T,\|h\|_{\mathbb{H}} \leq N} \varphi_{h}(t) & \leq \sup _{0 \leq t \leq T,\|h\|_{\mathbb{H}} \leq N}|x-y|^{\exp \left\{-\int_{0}^{t} C\left(1+|\dot{h}_1({u})|\right) \d u\right\}} \\
			& \leq|x-y|^{\exp \{-C T(1+N)\}}
		\end{aligned}
	\end{equation}
		for \( |x-y| \leq \eta^{\prime} \).
		By the linear growth of  \(\bar{f}_1 \) and \( \sigma_1 \), it is clear that there is a constant \( C=C(R, T, N) \) such that for all \( t,s \in[0, T] \)
	\begin{equation}\label{EQ13}
			\sup _{x \in D_{R},\|h\|_{\mathbb{H}} \leq N}|X_{t}^{h}(x)-X_{s}^{h}(x)| \leq C|t-s|^{1 / 2} .
	\end{equation}
		Both (i) and (ii) follow from \eqref{EQ11} and \eqref{EQ13}, by noticing that
		\[ 
		|X_t^h(x)| \leq |X_t^h(x)-X_0^h(x)|+|X_0^h(x)|, 
		 \]
		 and
		 \[ 
		 |X_t^h(x)-X_s^h(y)|\leq |X_t^h(x)-X_t^h(y)|+|X_t^h(y)-X_s^h(y)|
		  \]
    respectively. The proof is complete.
	\end{proof}
	
	\begin{lemma}\label{THM12}
	The mapping \( h \mapsto S(h) \) is continuous from \( B_{N} \) to \( \mathbb{C}_{T} \).
	\end{lemma}
	\begin{proof}
	Let \( h^{n}=(h_1^n,h_2^n) \rightarrow h=(h_1,h_2) \) weakly in \( B_{N} \). Set
		\[
		g_{n}(t, x):=\int_{0}^{t}\sigma(X_{s}^{h}(x))(\dot{h}^{n}_1({s})-\dot{h}_1({s})) \d s .
		\]
		Then for every \( (t, x) \in[0, T] \times D_{R} \),
		\[
		\lim _{n \rightarrow \infty}\left|g_{n}(t, x)\right|=0 .
		\]
		By the same method as in Lemma \ref{THM09}, we know that \( \left\{g_{n}\right\}_{n \in \mathbb{N}} \) is relatively compact in \( \mathbb{C}_{T}^{R} \), and so
		\begin{equation}\label{EQ14}
				\lim _{n \rightarrow \infty}\left\|g_{n}\right\|_{\mathbb{C}_{T}^{R}}=0.
		\end{equation}
		Put
		\[
		\psi_{n}(t):=\sup _{(s, x) \in[0, t] \times D_{R}}|X_{s}^{h^{n}}(x)-X_{s}^{h}(x)| .
		\]
		By \eqref{EQ06}, we have
		\begin{align*}
			\psi_{n}(t) &  \leq \left\|g_{n}\right\|_{\mathbb{C}_{T}^{R}}+\sup _{(s, x) \in[0, t] \times D_{R}}\left|\int_{0}^{s}(\sigma_1(X_{u}^{h^{n}}(x))-\sigma_1(X_{u}^{h}(x))) \dot{h}_1^{n}({u})\d u\right| \\
			& \quad+\sup _{(s, x) \in[0, t] \times D_{R}}\left|\int_{0}^{s}(\bar{f}_1(X_{u}^{h^{n}}(x))-\bar{f}_1(X_{u}^{h}(x))) \d u\right| \\
			& \leq \left\|g_{n}\right\|_{\mathbb{C}_{T}^{R}}+\int_{0}^{t} \sqrt{| X_{s}^{h^{n}}(x)- X_{s}^{h}(x)|^2\cdot g_1(| X_{s}^{h^{n}}(x)- X_{s}^{h}(x)|)}\cdot |\dot{h}_1^{n}({s})|\d s\\
			& \quad+C\int_{0}^{t} \rho_{\eta_1}(| X_{s}^{h^{n}}(x)- X_{s}^{h}(x)|)\d s \\
			& \leq\left\|g_{n}\right\|_{\mathbb{C}_{T}^{R}}+C\int_{0}^{t} \rho_{\eta^{\prime}}(\psi_{n}(s)) (1+|\dot{h}_1^{n}(s)|) \d s,
		\end{align*}
		where $\eta^{\prime}$ is given in \eqref{eq:08021}.	By Theorem \ref{THM02}-(2), we obtain
		\[
		\begin{aligned}
			\|X^{h^{n}}-X^{h}\|_{\mathbb{C}_{T}^{R}}=\psi_{n}(T) & \leq(\left\|g_{n}\right\|_{\mathbb{C}_{T}^{R}})^{\exp \left\{-C \int_{0}^{T}\left(1+|\dot{h}_1^{n}({u})|\right) \d u\right\}} \\
			& \leq(\left\|g_{n}\right\|_{\mathbb{C}_{T}^{R}})^{\exp \{-C T(1+N)\}},
		\end{aligned}
		\]
		which, together with \eqref{EQ14}, completes the proof.
	\end{proof}
	Based on the above two lemmas, we can obtain the following result. \begin{proposition}
		(i) For any \( f \in \mathbb{C}_{T} \), if \( I(f)<\infty \), there is a \( h_{0} \in \mathbb{H} \) such that \( 2 I(f)=\left\|h_{0}\right\|_{\mathbb{H}}^{2} \).\\
		(ii) \( I(f) \) is a rate function on \( \mathbb{C}_{T} \).
\end{proposition}
	\begin{proof}
		\begin{enumerate}[(i)]
			\item By the definition of \( I(f) \), there is a sequence \( \{h_{n} \}\subset \mathbb{H} \) such that \( \left\|h_{n}\right\|_{\mathbb{H}}^{2} \downarrow \) \( 2 I(f) \) and \( S\left(h_{n}\right)=f \). Put \( N:=\sup _{n}\left\|h_{n}\right\|_{\mathbb{H}} \). Note that  $ B_N $ is  relatively compact, then there is a subsequence \( \{h_{n_{k}} \}\) and \( h_{0} \) such that \( h_{n_{k}} \rightarrow h_{0} \) in \( B_{N} \). So \( \|h_{0}\|_{\mathbb{H}}^{2} \leq \liminf _{k \rightarrow \infty}\left\|h_{n_{k}}\right\|_{\mathbb{H}}^{2}=2 I(f) \). By Lemma \ref{THM12},  we have \( S\left(h_{0}\right)=f \), hence \( 2 I(f)=\left\|h_{0}\right\|_{\mathbb{H}}^{2} \).
			
			\item For any \( a<\infty \), obviously \( A:=\{f: I(f) \leq a\} \subset\left\{S(h) ;\|h\|_{\mathbb{H}}^{2} \leq 2 a\right\} \). By Lemma \ref{THM09}, notice that the closed subset of relatively compact set is compact, so we only need to prove that \( A \) is closed in \( \mathbb{C}_{T} \). Let \( A \ni f_{n} \rightarrow f \) in \( \mathbb{C}_{T} \). By (i) we can choose \( h_{n} \in B_{2 a} \) such that \( S\left(h_{n}\right)=f_{n} \). By the compactness of \( B_{2 a} \), there is a subsequence \( \{h_{n_{k}} \}\) and \( h \in B_{2 a} \) such that \( h_{n_{k}} {\rightarrow} h \) in \( B_{2 a} \). By Lemma \ref{THM12} we have \( f_{n_{k}}=S(h_{n_{k}}) \rightarrow S(h) \), hence \( f=S(h) \) and \( A \) is closed.
		\end{enumerate}
\end{proof}

	\subsection{Weak convergence}
	Our core goal is proving the weak convergence of controlled slow processes. To proceed,
we first give analysis of tightness which is crucial in the proof of the Laplace principle. Here, we list some lemmas for the modification of random fields with multiple parameters cited from Kunita \cite{Kunita1997,Kunita2019}. The first is the continuous modification criterion for the random fields.
	\begin{lemma}[Kolmogorov-Totoki's criterion]\label{THM11}
	Let \( \{X(x), x \in \mathbb{D}\} \) be a random field with values in a separable Banach space \(S\) with norm $ \|\cdot\|$, where \( \mathbb{D} \) is a bounded domain in \(\mathbb{R}^{d}\). Assume that there exist positive constants \(\gamma,C\) and \(\alpha>d\) satisfying
	\[
	\E[\|X(x)-X(y)\|^{\gamma}] \leq C|x-y|^{\alpha}, \quad \forall x, y \in \mathbb{D} .
	\]
	Then there exists a continuous random field \( \{\tilde{X}(x), x \in \bar{\mathbb{D}}\} \) such that \( \tilde{X}(x)=X(x) \) holds a.s., for any \( x \in \mathbb{D} \), where \( \bar{\mathbb{D}} \) is the closure of \( ~\mathbb{D} \). Furthermore, if \(~ 0 \in \mathbb{D} \) and \( \E\|X(0)\|^{\gamma}<\infty \), we have \( \E[\sup _{x \in \mathbb{D}}\|\tilde{X}(x)\|^{\gamma}]<\infty \).
	\end{lemma}
	
	The next lemma is the tightness criterion for a sequence of continuous stochastic fields.
	\begin{lemma}[Kolmogorov's  tightness criterion]\label{THM13}
		Let $ \{X_n(x):x\in \mathbb{I}\} $ be a sequence of continuous random fields with value in a real sperable semi-reflexive Frechet space $ S $ with seminorms $ \|\cdot\|_N,N\geq 1.$ Assume that for each $ N $ there exsits positive constants $ \gamma,C $ and $ \alpha_1,\cdots,\alpha_d $ with  $ \sum_{i=1}^{d}\alpha_i^{-1} <1 $ such that 
		\[ 
		\E[\|X_n(x)-X_n(y)\|_N^\gamma]\leq C\sum_{i=1}^{d}|x_i-y_i|^{\alpha_i},\quad \forall x,y \in \mathbb{I},
		\]   
		and 
		\[ 
		\E\|X_n(x)\|_N^{\gamma} \leq C,\quad \forall x\in \mathbb{I},
		\]
		holds for any $ n.$  Then $ \{X_n\} $  is tight with respect to the semi-weak topology of $C(\mathbb{I};S), $ the set consisted by all continuous maps from \( \mathbb{I} \) to \( S \).
	\end{lemma}

	Combining the prior estimation of $ \hat{X}_t^{(\varepsilon,\delta)} $ with above Kolmogorov-Totoki's criterion, we have the following result.
	\begin{lemma}\label{THM16}
		Under the assumption (\hyperlink{(H1)}{H1})-(\hyperlink{(H3)}{H3}), the slow variable of generalized system \eqref{EQ12} has a bi-continuous modification which belongs to $ \C_T, $ and still is denoted by $ \hat{X}_t^{(\varepsilon,\delta)}(x). $ 
		In other words, if  we fix the initial data of fast variable $ \hat{Y}^{(\varepsilon,\delta)} $ to be $ y_0, $  then the mapping
		\[ 
		(t,x) \mapsto \hat{X}_t^{(\varepsilon,\delta)}(x) \in \C_T\quad \P\text{-a.s.}. 
		 \]
	\end{lemma}
	\begin{proof}
	  Let \( f \) be a function as in Proposition \ref{THM18}. Since
		\[
		f(x+y) \leq f(x)+f(y), \quad \forall x, y \geq 0,
		\]
	we can define a new complete metric on \( \mathbb{R}^{d} \) by means of 
		\[
		\operatorname{dist}(x, y):=f(|x-y|).
		\]
		Then by Lemma \ref{THM17} and Proposition \ref{THM18}, there exist two positive constants \( p^{\prime}\) and \( C^{\prime} \) such that
	\begin{equation}\label{EQ51}
\mathbb{E}\left(\operatorname{dist}\big(\hat{X}_{t}^{(\varepsilon,\delta)}(x_1), \hat{X}_{s}^{(\varepsilon,\delta)}(x_2)\big)\right)^{p^{\prime}} \leq C^{\prime }(|t-s|+|x_1-x_2|)^{2(d+1)}
	\end{equation}
		for all \( \varepsilon \in\left[0, \varepsilon_{0}\right], s, t \in[0, T] \) and \( x_1, x_2 \in D_{R}\). By Lemma \eqref{THM11}, we know that \( \hat{X}^{(\varepsilon, \delta)} \in \mathbb{C}_{T}^{R}~\)  for all \( R \in \mathbb{N} \), and therefore belongs to \( \mathbb{C}_{T} \).
	\end{proof}
Based on the above lemma, the following tightness results can be obtained.

\begin{theorem}\label{THM10}
	(i) For every \( \varepsilon \in\left[0, \varepsilon_{0}\right]\) and \(y_0\in \R^d\),   fixing the initial data of the fast variable $ \hat{Y}^{(\varepsilon,\delta)} $ to be $ y_0, $  \( \hat{X}_{\cdot}^{(\varepsilon,\delta)}(\cdot)\in \mathbb{C}_{T} \).
	
	(ii) The laws of \( \{(h^{(\varepsilon,\delta)}, \hat{X}^{(\varepsilon,\delta)}, W), \varepsilon \in\left[0, \varepsilon_{0}\right]\} \)  in \( B_{N} \times \mathbb{C}_{T} \times C([0,T];\R^{2d}) \) is tight.
	
	(iii) There exists a probability space \( (\bar{\Omega}, \bar{\mathscr{F}}, \bar{\P}) \) and a sequence (still indexed by $ \varepsilon $
	for simplicity) \( \{(\bar{h}^{(\varepsilon,\delta)}, \bar{X}^{(\varepsilon,\delta)}, \overline{W}^{(\varepsilon,\delta)})\} \) and \( (\bar{h}, \bar{X}^{\bar{h}}, \overline{W}) \) defined on this probability space which taking values in \( B_{N} \times \mathbb{C}_{T} \times \mathcal{W}_{T} \) such that
	
	(a) \( (\bar{h}^{(\varepsilon,\delta)}, \bar{X}^{(\varepsilon,\delta)}, \overline{W}^{(\varepsilon,\delta)})\) has the same law as  \( \{(h^{(\varepsilon,\delta)}, \hat{X}^{(\varepsilon,\delta)}, W)\} \)   for each \( \varepsilon \).
	
	(b) \( (\bar{h}^{(\varepsilon,\delta)}, \bar{X}^{(\varepsilon,\delta)}, \overline{W}^{(\varepsilon,\delta)}) \stackrel{a.s.}{\longrightarrow}(\bar{h}, \bar{X}^{\bar{h}}, \overline{W})  \) in \( B_{N} \times \mathbb{C}_{T} \times \mathcal{W}_{T} \), \( \varepsilon \rightarrow 0 \).
	
	(c) \( (\bar{h}, \bar{X}^{\bar{h}}) \) solves the following ODE:
	\begin{equation}\label{EQ18}
		\bar{X}_{t}^{\bar{h}}(x)=\int_{0}^{t}\bar{f}_{1}({\bar{X}}_{s}^{\bar{h}}(x)) \d s+\int_{0}^{t}\sigma_{1}({\bar{X}}_{s}^{\bar{h}}(x)) \dot{\bar{h}}_1(s) \d s.
	\end{equation}
	for any \( t \in[0, T] \) and \( x \in \mathbb{R}^{d} \), where $\bar{f}_{1}(x)  $ is the averaging coefficient of  $ f_1(x,y) $, that is $ \bar{f}_{1}(x)=\int_{\mathbb{R}^{{d}}} f_{1}(x, {y}) \mu_{{x}}(\d {y}).  $
\end{theorem}
\begin{proof}
	(i) The result follows from Lemma \ref{THM16}.
	%
	%
	
	(ii) The tightness of the laws of  \( \{(h^{(\varepsilon,\delta)}, \hat{X}^{(\varepsilon,\delta)}, W)\} \) in  \( B_{N} \times \mathbb{C}_{T} \times C([0,T];\R^{2d})\), follows from \eqref{EQ51}, Lemma \ref{THM13},  and  \( \sup \|h^{(\varepsilon,\delta)}\|_{\mathbb{H}} \leq N \), $ \P $-a.s..
	
	(iii)  By (ii), Prokhorov's theorem and Skorohod's representation theorem, we deduce that there exist a probability space \( (\bar{\Omega}, \bar{\mathscr{F}}, \bar{\mathbb{P}}) \), and   random variables \( (\bar{h}^{(\varepsilon,\delta)}, \bar{X}^{(\varepsilon,\delta)}, \overline{W}^{(\varepsilon,\delta)}) \),    \((h, X^{h}, \overline{W}) \)  on this space such that assertions (a), (b) hold. Note that we have 
 \[
 \bar{h}^{(\varepsilon,\delta)}=(\bar{h}_1^{(\varepsilon,\delta)},\bar{h}_2^{(\varepsilon,\delta)}),\overline{W}^{(\varepsilon,\delta)}=(\overline{W}^{1,(\varepsilon,\delta)},\overline{W}^{2,(\varepsilon,\delta)}).
 \]
 The following calculation will be treated on the space \( (\bar{\Omega}, \bar{\mathscr{F}}, \bar{\mathbb{P}}) \).
 For (c), noticing that  \( \overline{W}^{(\varepsilon,\delta)} \) is a Brownian motion on \( (\bar{\Omega}, \bar{\mathscr{F}}, \bar{\P}) \)  and has the same law as \( W \) for each \( \varepsilon \).  Let	$ X^{(\varepsilon,\delta)} $ be  the slow component of \eqref{EQ:SDE:01}.
	There is a $ \mathscr{F}\backslash\mathcal{B}(\C_T) $-measurable function 
	\[ 
	\Phi^{(\varepsilon,\delta)} :\Omega\to \C_T
	\]
	such that
	\begin{equation}\label{EQ49}
X^{(\varepsilon,\delta)}= \Phi^{(\varepsilon,\delta)}(W)(t,x).
	\end{equation}
	Hence,
	\[	\hat{X}^{(\varepsilon,\delta)}=\Phi^{(\varepsilon,\delta)}(W+\frac{1}{\sqrt{\varepsilon}}h^{(\varepsilon,\delta)}) \quad \P\text {-a.s.},
	\]
	which, together with (a), implies
	\[	\bar{X}^{(\varepsilon,\delta)}=\Phi^{(\varepsilon,\delta)}(\overline{W}^{(\varepsilon,\delta)}+\frac{1}{\sqrt{\varepsilon}}\bar{h}^{(\varepsilon,\delta)}), \quad \bar{\P}\text {-a.s..}
	\]
	Now by virtue of the uniqueness of strong solution and the Girsanov theorem, \( \bar{X}^{(\varepsilon,\delta)}\) is the slow component of solution  \( (\bar{X}^{(\varepsilon,\delta)},\bar{Y}^{(\varepsilon,\delta)})\) to the following equation:
	\begin{equation}\label{EQ47}
		\left\{\begin{aligned}
			\d \bar{X}^{(\varepsilon,\delta)}& =  f_{1}(\bar{X}^{(\varepsilon,\delta)}, \bar{Y}^{(\varepsilon,\delta)}) \d t+\sigma_{1}(\bar{X}^{(\varepsilon,\delta)}) \dot{\bar{h}}_1^{(\varepsilon, \delta)} \d t+\sqrt{\varepsilon} \sigma_{1}(\bar{X}^{(\varepsilon,\delta)}) \d \overline{W}_{t}^{1,(\varepsilon,\delta)} ,\\
			\delta \d \bar{Y}^{(\varepsilon,\delta)}& =  f_{2}(\bar{X}^{(\varepsilon,\delta)}, \bar{Y}^{(\varepsilon,\delta)}) \d t+\sqrt{\frac{\delta}{\varepsilon}} \sigma_{2}(\bar{X}^{(\varepsilon,\delta)}, \bar{Y}^{(\varepsilon,\delta)}) \dot{\bar{h}}_2^{(\varepsilon, \delta)} \d t+\sqrt{\delta} \sigma_{2}(\bar{X}^{(\varepsilon,\delta)}, \bar{Y}^{(\varepsilon,\delta)}) \d \overline{W}_{t}^{2,(\varepsilon,\delta)}.
		\end{aligned}\right.
	\end{equation}
 Note that $\bar{h}^{(\varepsilon,\delta)}\stackrel{\bar{\P}\text{-a.s.}}{\longrightarrow}\bar{h},\varepsilon\to 0 $. Let $X^{\bar{h}}$ be the corresponding solution to \eqref{EQ06}. Thus,  we have 
 \begin{equation}\label{EQ:SKETLON:100}
\begin{split}		
|X_t^{\bar{h}^{(\varepsilon,\delta)}}-X_t^{\bar{h}}|  & \leq \left|\int_{0}^{t}\bar{f}_1(X_s^{\bar{h}^{(\varepsilon,\delta)}})-\bar{f}_1(X_s^{\bar{h}})\d s\right|+ \left|\int_{0}^{t}(\sigma_1(X_s^{\bar{h}^{(\varepsilon,\delta)}})-\sigma_1(X_s^{\bar{h}}))\dot{\bar{h}}_1^{(\varepsilon,\delta)}(s)\d s\right|\\
		&\quad  + \left|\int_{0}^{t}\sigma_1(X_s^h)(\dot{\bar{h}}_1^{(\varepsilon,\delta)}(s)-\dot{{\bar{h}}}_1(s))\d s\right| \\
   &=: G_1+G_2+G_3.
\end{split}
\end{equation}
By \eqref{EQ27}, we have
\begin{equation}\label{EQ:SKETLON:aa1}
	G_1\leq \int_{0}^{t}|X_s^{\bar{h}^{(\varepsilon,\delta)}}-X_s^{\bar{h}}| \gamma(|X_s^{\bar{h}^{(\varepsilon,\delta)}}-X_s^{\bar{h}}|\leq \int_{0}^{t} \rho_{\eta_1}(|X_s^{\bar{h}^{(\varepsilon,\delta)}}-X_s^{\bar{h}}|)\d s.
\end{equation}
Due to the assumption (\hyperlink{(H1)}{H1})  and \hold's inequality,
\begin{equation}\label{EQ:SKETLON:102}
	G_2 \leq \left[\int_{0}^{t}|X_s^{\bar{h}^{(\varepsilon,\delta)}}-X_s^{\bar{h}}|^2\cdot g_1(|X_s^{\bar{h}^{(\varepsilon,\delta)}}-X_s^{\bar{h}}|)\d s\right]^{1/2} \left[\int_{0}^{t}|\dot{\bar{h}}_1^{(\varepsilon,\delta)}(s)|^2\d s\right]^{1/2}.
\end{equation}
 By the linearity of $ \sigma_1(\cdot) $ and \hold's inequality, we have
\begin{equation}\label{EQ:SKETLON:103}
	 G_3\leq \left[\int_{0}^{t}(1+|X_s^{\bar{h}}|^2)\d s\right]^{1/2}\left[\int_{0}^{t}|\dot{\bar{h}}_1^{(\varepsilon,\delta)}(s)-\dot{{\bar{h}}}_1(s)|^2\d s\right]^{1/2}.
\end{equation}
Substituting \eqref{EQ:SKETLON:aa1}-\eqref{EQ:SKETLON:103} into  \eqref{EQ:SKETLON:100}, by Jensen's inequality and  Lemma \ref{THM02}-(2), we also have
  \[ 
 \limsup_{\varepsilon \to0 } \E_{\bar{\P}} [\sup_{0 \leq t \leq T,x\in D_N}|X_t^{\bar{h}^{(\varepsilon,\delta)}}(x)-X_t^{\bar{h}}(x)| ] =0.
   \]
 By   \eqref{EQ47}, we have
   \[ 
   \begin{split}
   	\d (\bar{X}_t^{(\varepsilon,\delta)}-X_t^{\bar{h}^{(\varepsilon,\delta)}})&=(f_1(\bar{X}_t^{(\varepsilon,\delta)},\bar{Y}_t^{(\varepsilon,\delta)})-\bar{f}_1(X_t^{\bar{h}^{(\varepsilon,\delta)}})) \d t+(\sigma_1(\bar{X}_t^{(\varepsilon,\delta)})-\sigma_1(X_t^{\bar{h}^{(\varepsilon,\delta)}}))\bar{h}_1^{(\varepsilon,\delta)}(t)\d t\\
   	& \quad +\sqrt{\varepsilon}\sigma_1(\bar{X}_t^{(\varepsilon,\delta)})\d \overline{W}_{t}^{1,(\varepsilon,\delta)}.
   \end{split}
    \]
 Then by the similar argument in the  proof of Lemma \ref{THM18} and the aforementioned convergence about skeleton process sequences $\{X^{\bar{h}^{(\varepsilon,\delta)}}\}$, we can deduce that for any $\iota>0$
 \[
 \lim_{\varepsilon\to 0}\bar{\P}[\sup_{0 \leq t \leq T,x\in D_N}|\bar{X}_t^{(\varepsilon,\delta)}-{X}_t^{\bar{h}}|>\iota] =0,\quad \varepsilon\to0.
 \]
Hence, uniqueness of  the limit convergence implies the desired result.
\end{proof}
			
\subsection{Proof of Laplace principle}
This section mainly proves the Laplace principle, which is equivalent to the LDP,   by proving the upper and lower bounds of the Laplace principle, respectively. To this aim, we give the variational representation formula about the slow variable $X^{(\varepsilon,\delta)}$ firstly.

For any $ h=(h_1,h_2)\in \mathcal{A},$  let $(X^{(\varepsilon,\delta),h},Y^{(\varepsilon,\delta),h}) $ be the solution of the following controlled system,
	\begin{equation}\label{EQ20}
		\left\{\begin{aligned}
			\d {X}_{t}^{(\varepsilon, \delta),h}&=  f_{1}({X}_{t}^{(\varepsilon, \delta),h}, {Y}_{t}^{(\varepsilon, \delta),h}) \d t+\sigma_{1}({X}_{t}^{(\varepsilon, \delta),h}) \dot{h}_1(t) \d t+\sqrt{\varepsilon} \sigma_{1}({X}_{t}^{(\varepsilon, \delta),h}) \d W^1_{t} \\
			\delta \d {Y}_{t}^{(\varepsilon, \delta),h}&=  f_{2}({X}_{t}^{(\varepsilon, \delta),h}, {Y}_{t}^{(\varepsilon, \delta),h}) \d t+\sqrt{\frac{\delta}{\varepsilon}} \sigma_{2}({X}_{t}^{(\varepsilon, \delta),h}, {Y}_{t}^{(\varepsilon, \delta),h}) \dot{h}_2(t)\d t\\
   &\quad+\sqrt{\delta} \sigma_{2}({X}_{t}^{(\varepsilon, \delta),h}, {Y}_{t}^{(\varepsilon, \delta),h}) \d W^2_{t}.
		\end{aligned}\right.
	\end{equation}
	\begin{lemma}[Variational presentation formula]\label{THM08}
		The slow variable $ X^{(\varepsilon,\delta)}  $ of the system \eqref{EQ:SDE:01} satisfies the following variational formula, for any bounded continuous  \( F: \C_T \rightarrow \mathbb{R},\)
		\[
		-\varepsilon \log \mathbb{E}\left[\exp \left(-\frac{F(X^{(\varepsilon, \delta)})}{\varepsilon} \right)\right]=\inf _{h \in \mathcal{A}} \mathbb{E}\left[\frac{1}{2} \int_{0}^{T}|\dot{h}(s)|^{2} \d s+F(X^{(\varepsilon, \delta), h})\right],
		\]
	where $X^ {(\varepsilon,\delta), h} $ is the slow variable of the controlled system \eqref{EQ20}.
	\end{lemma}
	\begin{proof}
  By \eqref{EQ49},  the slow variable $X^{(\varepsilon, \delta), \sqrt{\varepsilon}h}$ of the controlled system \eqref{EQ20} can be expressed as  $\Phi ^ {(\varepsilon, \delta)} (W+h)$.
  It can be easily seen from the well-known variational presentation formula of Brownian motion (cf. Bou\'{e} and Dupuis \cite[Theorem 3.1]{DupuisAOP1998}) that
		\begin{equation}
			\begin{split}
				-\log \E\left[\exp(-g(X^{(\varepsilon,\delta)}))\right]  & =-\log \E\left[\exp(-g\circ \Phi^{(\varepsilon,\delta)}(W))\right]\\
				& = \inf_{h\in \mathcal{A}}\E\left[g\circ \Phi^{(\varepsilon,\delta)}(W+h)+\frac{1}{2}\|h\|^2_{\mathbb{H}}\right] \\
				& =  \inf_{h\in \mathcal{A}}\E\left[g(X^{(\varepsilon,\delta),\sqrt{\varepsilon}h})+\frac{1}{2}\|h\|^2_{\mathbb{H}}\right] .
			\end{split}
		\end{equation}
	In particular, we can take  $g (\cdot):=\frac{F (\cdot)} {\varepsilon}, $ then
	\begin{align*}
			-\varepsilon \log \mathbb{E}\left[\exp \left(-\frac{F(X^{(\varepsilon, \delta)})}{\varepsilon} \right)\right]& = -\varepsilon\log \E[\exp(-g(X^{(\varepsilon,\delta)}))]  \\
			& =  \varepsilon\inf_{h\in \mathcal{A}}\E\left[g(X^{(\varepsilon,\delta),\sqrt{\varepsilon}h})+\frac{1}{2}\|h\|^2_{\mathbb{H}}\right] \\
			& = 	\varepsilon\inf_{h\in \mathcal{A}}\E\left[\frac{1}{\varepsilon}F(X^{(\varepsilon,\delta),\sqrt{\varepsilon}h})+\frac{1}{2}\|h\|^2_{\mathbb{H}}\right] \\
			& = \inf_{h\in \mathcal{A}}\E\left[F(X^{(\varepsilon,\delta),\sqrt{\varepsilon}h})+\frac{\varepsilon}{2}\|h\|^2_{\mathbb{H}}\right] \\
			& = \inf_{h\in \mathcal{A}}\E\left[F(X^{(\varepsilon,\delta),\sqrt{\varepsilon}h})+\frac{1}{2}\|\sqrt{\varepsilon}h\|^2_{\mathbb{H}}\right]\\
			& = \inf_{h^{\prime}\in \mathcal{A}}\E\left[F(X^{(\varepsilon,\delta),h^{\prime}})+\frac{1}{2}\|h^{\prime}\|^2_{\mathbb{H}}\right] \\
			&= \inf _{h \in \mathcal{A}} \mathbb{E}\left[\frac{1}{2} \int_{0}^{T}|\dot{h}(s)|^{2} \d s+F(X^{(\varepsilon, \delta), h})\right] .
	\end{align*}
	The proof is complete.	
	\end{proof}
 Now, we can give the proof of main result.
 \begin{proof}[Proof of Theorem \ref{THM22}]
	(Lower bound). Let \(g\) be a real bounded continuous function on \( \mathbb{C}_{T} \). For $ X^{(\varepsilon,\delta)} $ in the origin slow-fast motion \eqref{EQ:SDE:01},
    by Lemma \ref{THM08}, we have
	\begin{equation}\label{EQ:VARITIONAL:01}
	    \begin{aligned}
		-\varepsilon \log \mathbb{E}\left[\exp \left(-\frac{g\left(X^{(\varepsilon,\delta)}\right)}{\varepsilon}\right)\right] & =\inf _{h \in \mathcal{A}} \mathbb{E}\left[g(X^{(\varepsilon,\delta), h})+\frac{1}{2}\|h\|_{\mathbb{H}}^{2}\right] .
	\end{aligned}
	\end{equation}
	Fixing \( \theta>0 \), by the definition of infimum, there is a  \( h^{(\varepsilon,\delta)}=(h_1^{(\varepsilon,\delta)},h_2^{(\varepsilon,\delta)}) \in \mathcal{A} \) such that
	\[
	\inf _{h \in \mathcal{A}} \mathbb{E}\left[g(X^{(\varepsilon,\delta), h})+\frac{1}{2}\|h\|^2_{\mathbb{H}}\right] \geq \mathbb{E}\left[g(X^{(\varepsilon,\delta), h^{(\varepsilon,\delta)}})+\frac{1}{2}\|h^{(\varepsilon,\delta)}\|_{\mathbb{H}}^{2}\right]-\theta,
	\]
	for every pair \( 0<\delta<\varepsilon<1 \). Since \( g \)  is bounded, we have
	\[
	\frac{1}{2} \sup _{\varepsilon>0} \mathbb{E}\left[\|h^{(\varepsilon,\delta)}\|_{\mathbb{H}}^{2}\right] \leq 2\|g\|_{\infty}+\theta.
	\]
	 Define
	\[
	\tau_{N}^{(\varepsilon,\delta)}:=\inf \left\{t \in[0, T]: \int_{0}^{t}|\dot{h}^{(\varepsilon,\delta)}(s)| ^2\d s \geq N\right\},
	\]
	and
	\[
	h_{N}^{(\varepsilon,\delta)}(t):=h^{(\varepsilon,\delta)}(t \wedge \tau_{N}^{(\varepsilon,\delta)}).
	\]
	Then \( h_{N}^{(\varepsilon,\delta)}(t) \in \mathcal{A}_{N} \) and
	\[
	\P(\|h_{N}^{(\varepsilon,\delta)}-h^{(\varepsilon,\delta)}\|_{\mathbb{H}} \neq 0)=\P(\tau_{N}^{(\varepsilon,\delta)}<T) \leq \cfrac{2\left(2\|g\|_{\infty}+\theta\right)}{N} .
	\]
	Therefore,
	\[
	\begin{split}
	    	\mathbb{E}&\left[g(X^{(\varepsilon,\delta), h^{(\varepsilon,\delta)}})+\frac{1}{2}\|h^{(\varepsilon,\delta)}\|_{\mathbb{H}}^{2}\right]-\theta \\
      & \geq \mathbb{E}\left[g(X^{(\varepsilon,\delta), h_{N}^{(\varepsilon,\delta)}})+\frac{1}{2}\|h_{N}^{(\varepsilon,\delta)}\|_{\mathbb{H}}^{2}\right]-\cfrac{2\|g\|_{\infty}(2\|g\|_{\infty}+\theta)}{N}-\theta .
	\end{split}
	\]
	For $h_N^{(\varepsilon,\delta)}$, applying Theorem \ref{THM10}, there exist a probability space \( (\bar{\Omega}, \bar{\mathscr{F}}, \bar{\P}) \) and a sequence (still indexed by $ \varepsilon $
	for simplicity) \( \{(\bar{h}_N^{(\varepsilon,\delta)}, \bar{X}_N^{(\varepsilon,\delta)}, \overline{W}^{(\varepsilon,\delta)})\} \) and \( (\bar{h}_N, \bar{X}_N^{\bar{h}_N}, \overline{W}) \) defined on this probability space which taking values in \( B_{N} \times \mathbb{C}_{T} \times \mathcal{W}_{T} \) such that (iii)-(a), (b) and (c) in Theorem \ref{THM10} hold. By Fatou's lemma, we have
	\begin{equation*}
		\begin{split}
			\liminf _{\varepsilon \rightarrow 0} \mathbb{E}\left[g(X^{(\varepsilon,\delta), h_{N}^{(\varepsilon,\delta)}})+\frac{1}{2}\|h_{N}^{(\varepsilon,\delta)}\|_{\mathbb{H}}^{2}\right]
			& =\liminf _{\varepsilon \rightarrow 0} \mathbb{E}_{\bar{\P}}\left[g(\bar{X}_N^{(\varepsilon,\delta)})+\frac{1}{2}\|\bar{h}_{N}^{(\varepsilon,\delta)}\|_{\mathbb{H}}^{2})\right]\\
			& \geq \mathbb{E}_{\bar{\P}}\left[g(S(\bar{h}_N))+\frac{1}{2}\|\bar{h}_N\|_{\mathbb{H}}^{2}\right] \\
			& \geq \inf _{\left\{(f, h) \in \mathbb{C}_{T} \times \mathbb{H}: f=S(h)\right\}}\left(g(f)+\frac{1}{2}\|h\|_{\mathbb{H}}^{2}\right) \\
			&\geq \inf _{f \in \mathbb{C}_{T}}(g(f)+I(f)) .
		\end{split}
	\end{equation*}
	Therefore,
	\[
	\begin{split}
		\liminf _{\varepsilon \rightarrow 0}-\varepsilon \log \mathbb{E}\left[\exp \left(-\frac{g\left(X^{(\varepsilon,\delta)}\right)}{\varepsilon}\right)\right]
		\geq \inf _{f \in \mathbb{C}_{T}}(g(f)+I(f))-\cfrac{2\|g\|_{\infty}\left(2\|g\|_{\infty}+\delta\right)}{N}-\theta.
	\end{split}
	\]
	Finally, letting \( N \rightarrow \infty \) and \( \theta \rightarrow 0 \) yields the lower bound, 
	\begin{equation}\label{EQ21}
		\begin{split}
			\liminf _{\varepsilon \rightarrow 0}-\varepsilon \log \mathbb{E}\left[\exp \left(-\frac{g\left(X^{(\varepsilon,\delta)}\right)}{\varepsilon}\right)\right] 
			\geq \inf _{f \in \mathbb{C}_{T}}(g(f)+I(f)).
		\end{split}
	\end{equation}
	
	(Upper bound). Fix \( \theta>0 \). Since \( g \) is a bounded function, there exists a   \( f_{0} \in \mathbb{C}_{T} \) such that
	\[
	g\left(f_{0}\right)+I\left(f_{0}\right) \leq \inf _{f \in \mathbb{C}_{T}}(g(f)+I(f))+\theta,
	\]
	and then choose  \( h_{0} \in \mathbb{H} \) such that
	\[
	\frac{1}{2}\left\|h_{0}\right\|_{\mathbb{H}}^{2}=I\left(f_{0}\right) , f_{0}=S\left(h_{0}\right) .
	\]
	Similarly, for $h_0$, applying  Theorem \ref{THM10} and Fatou's lemma,  we have 
	\[
	\begin{split}
		\limsup _{\varepsilon \rightarrow 0}-\varepsilon \log \mathbb{E}\left[\exp \left(-\frac{g\left(X^{(\varepsilon,\delta)}\right)}{\varepsilon}\right)\right] 
		&=\limsup _{\varepsilon \rightarrow 0} \inf _{h \in \mathcal{A}} \mathbb{E}\left[g(X^{(\varepsilon,\delta), h})+\frac{1}{2}\|h\|_{\mathbb{H}}^{2}\right] \\
		&  \leq \limsup _{\varepsilon \rightarrow 0} \mathbb{E}\left[g(X^{(\varepsilon,\delta), h_{0}})+\frac{1}{2}\left\|h_{0}\right\|_{\mathbb{H}}^{2}\right] \\
		& \leq g\left(S\left(h_{0}\right)\right)+\frac{1}{2}\left\|h_{0}\right\|_{\mathbb{H}}^{2} \\
		& \leq \inf _{f \in \mathbb{C}_{T}}(g(f)+I(f))+\theta .
	\end{split}
	\]
where we  have used    \eqref{EQ:VARITIONAL:01} in the first equality. Letting \( \theta\rightarrow 0 \), we get the upper bound
	\begin{equation}\label{EQ22}
		\limsup _{\varepsilon \rightarrow 0}-\varepsilon \log \mathbb{E}\left[\exp \left(-\frac{g\left(X^{(\varepsilon,\delta)}\right)}{\varepsilon}\right)\right] \leq \inf _{f \in \mathbb{C}_{T}}(g(f)+I(f)).
	\end{equation}	
	
 Combine  \eqref{EQ21} with \eqref{EQ22}, we can deduce the Laplace principle,
	\[ 
	\lim _{\varepsilon \rightarrow 0}-\varepsilon \log \mathbb{E}\left[\exp \left(-\frac{g\left(X^{(\varepsilon,\delta)}\right)}{\varepsilon}\right)\right]   =\inf _{f \in \mathbb{C}_{T}}(g(f)+I(f)).
	\]
 The proof is complete.
	\end{proof}
\bibliographystyle{plain}
\bibliography{reference_LDP}
\urlstyle{same}
	
\end{document}